\documentclass{amsart}
\usepackage{amssymb,hyperref,amsmath,bm,hyperref}

\usepackage[round]{natbib} 

\usepackage{xypic}
\newtheorem{definition}{Definition}[section]
\newtheorem{lemma}[definition]{Lemma}
\newtheorem{claim}[definition]{Claim}
\newtheorem{remark}[definition]{Remark}
\newtheorem{fact}[definition]{Fact}
\newtheorem{theorem}[definition]{Theorem}
\newtheorem{corollary}[definition]{Corollary}

\def\g{\mathrm{minimal}}		% minimal  move

\def\plim{{\mathrm{lim}}}	% limit operation of pca.
\def\alabel#1{\label{#1}}
\def\B{{\mathcal{B}}}		% PCA

\def\dual#1{{#1}^\perp}

\def\d{\mathbf{d}}

\def\defined{\downarrow}

\def\Nset{\mathbb{N}}
\def\num#1{{\overline{#1}^\A}}

\def\yohji#1{[{#1}]_\sim}	% congruence class

\def\IQC{\mathsf{IQC}}		% The intuitionistic predicate calculus

\def\A{\mathcal{A}}

\def\Repr{\mathrm{RpFn}}
\def\Nset{\mathbb{N}}

\def\pa{\mathsf{PA}}		% PA with all the names of primitive
\def\ha{\mathsf{HA}}		% HA with all the names of primitive
\def\pra{\ha}			% Then language of arithmetic with all
\def\eon{\mathsf{EON}}

\def\k{\mathbf{k}}		% the 1st combinatory constant
\def\s{\mathbf{s}}		% the 2nd combinatory constant
\def\cased{\mathbf{d}}
\def\suc{\mathbf{s_N}}
\def\pred{\mathbf{p_N}}
\def\realize#1#2{#1\;\mathbf{r}\;#2}

\def\p{\mathbf{p}}
\def\car{\mathbf{p_0}}
\def\cdr{\mathbf{p_1}}

\def\lequiv{\leftrightarrow}

\def\dne#1{\Sigma^0_{#1}\mbox{-}\mathbf{DNE}}
\def\pdne#1{\Pi^0_{#1}\mbox{-}\mathbf{DNE}}

\def\slem#1{\Sigma^0_{#1}\mbox{-}\mathbf{LEM}}
\def\plemm#1{\Pi^0_{#1}\mbox{-}\mathbf{LEM}'}
\def\slemm#1{\Sigma^0_{#1}\mbox{-}\mathbf{LEM}'}
\def\plem#1{\Pi^0_{#1}\mbox{-}\mathbf{LEM}}
\def\dlem#1{\Delta^0_{#1}\mbox{-}\mathbf{LEM}}
\def\fpdlem#1{\mathsf{fp}\Delta^0_{#1}\mbox{-}\mathbf{LEM}}
\def\bsdne#1{B\Sigma^0_{#1}\mbox{-}\mathbf{DNE}}
\def\sdne#1{\Sigma^0_{#1}\mbox{-}\mathbf{DNE}}
\def\sdneprime#1{\Sigma^0_{#1}\mbox{-}\mathbf{DNE}'}

\def\fip#1{F_{#1}\mbox{-}\mathbf{IP}}
\def\pllpo#1{(\Pi^0_{#1}\vee \Pi^0_{#1})\mbox{-}\mathbf{DNE}}
\def\sllpo#1{\Sigma^0_{#1}\mbox{-}\mathbf{LLPO}}

\def\dneg{\neg\neg}

\def\Q{\overline{Q}}		% quantifiers

\begin{document}

\title{Realizability Interpretation of PA by Iterated Limiting PCA}
\author[Y. Akama]{Yohji Akama}
\address{Mathematical Institute, Tohoku University, Aoba, Sendai, JAPAN, 980-8578.
 }

\maketitle
\begin{abstract} 
 For any partial combinatory algebra~(\textsc{pca} for short) $\A$, the class of $\A$-representable partial
functions from $\Nset$ to $\A$ quotiented by the filter of
cofinite sets of $\Nset$, is a \textsc{pca} such that the representable
partial functions are exactly the
limiting partial functions of $\A$-representable partial
functions~(Akama, ``Limiting partial combinatory algebras''
 Theoret. Comput. Sci. Vol.~311
2004).
The $n$-times iteration of this construction results in a \textsc{pca}
that represents any $n$-iterated limiting partial recursive functions,
and the inductive limit of the \textsc{pca}s over all $n$ is a
\textsc{pca} that represents any arithmetical, partial
function. Kleene's realizability interpretation over the former
\textsc{pca} interprets the logical principles of double negation
elimination for $\Sigma^0_n$-formulas, and that over the latter
\textsc{pca} interprets Peano's arithmetic~($\pa$ for short).  A
hierarchy of logical systems between Heyting's arithmetic and $\pa$ is
used to discuss the prenex normal form theorem, the relativized
 independence-of-premise schemes, and ``PA is an unbounded extension of HA.''
\end{abstract}
\section{Introduction}\alabel{sec:intro}

\subsection{Hierarchical of semi-classical arithmetical principles}
Following Section~1.3.2 of \cite{MR48:3699}, by \emph{Heyting's
arithmetic} $\ha$, we mean an intuitionistic predicate calculus $\IQC$
with equality such that (1) the language of $\ha$ is a first-order
language $L_{\pra}$ , with logical connectives $\forall, \exists, \to,
\wedge, \vee, \neg$; numeral variables $l,m,n,\ldots$; a constant symbol
$0$~(zero), a unary function symbol $S$~(successor), constant function
symbols for all primitive recursive functions, and a binary predicate
symbol $=$~(equality between numbers). \emph{Bounded quantifications}
$\forall n<t.\ A$ and $\exists n<t. \ A$ are abbreviations of $\forall
n(f(n,t)=1\to A)$ and $\exists n(f(n,t)=1 \wedge A)$, where $f(n,t)$ is
a primitive recursive function such that $f(n,t)=1$ if and only if
$n<t$; and (2) besides the axioms for the equality, the axioms of $\ha$
are the defining equality of the primitive recursive functions and
so-called Peano's axiom $\forall n(\neg S(n) = 0)$, $\forall n \forall m
(S(n)=S(m)\to n=m)$, and an axiom scheme called \emph{the induction
scheme}:
\begin{align*}
B [0] \wedge \forall n (B [n] \to B[S(n)])\to \forall n  B [n]\quad(\mbox{$B$ is any formula.})
\end{align*} 
By \emph{Peano's arithmetic} $\pa$, we mean the formal system obtained
from $\ha$ by adjoining one of classical axiom scheme, such as \emph{the
law of excluded middle} $A\vee\neg A$ ($A$ is any $L_\ha$-formula),
and/or \emph{the principle of double negation elimination} $\neg\neg
A\to A$ ($A$ is any $L_\ha$-formula). \cite{MR0015346} interpreted every
theorem of $\ha$ by a recursive function/operation.

Kleene introduced \emph{arithmetical hierarchy} of integer sets, over
the class of recursive sets. The complexity of an integer set $X$ in the
arithmetical hierarchy is measured by the number of alternation of
the quantifiers of the relation that defines the set $X$.  The
arithmetical hierarchy has a close relation to \emph{oracle}
computation, such as the \emph{complete sets} and the \emph{jump
hierarchy}~(see \cite{MR982269} for example).

According to Section~0.30 of \cite{MR1748522}, a $\Sigma^0_k$-formula and a
$\Pi^0_k$-formula are the following formulas preceded by $k$
alternating quantifiers, respectively for $k\ge0$:
\begin{itemize}
\item A $\Sigma^0_k$-formula is of the form $\exists {n_1}\forall {n_2}\cdots
  Q  {n_{k-1}} \Q  {n_k}.\ P [{n_1},\ldots, {n_k}]$.

 \item A $\Pi^0_k$-formula is of the form $
  \forall {m_1} \exists {m_2} \cdots \Q {m_{k-1}}  Q {m_k}.\  P [{m_1},\ldots, {m_k}]$.
\end{itemize}
Here $P[n_1, \ldots, n_k]$ is an $L_\pra$-formula with all the
quantifiers being bounded, but may contain free variables other than its
indicated variables. The $L_\pra$-formula $P[m_1, \ldots, m_k]$ is
understood similarly.  

A formula in \emph{prenex normal form}~(\textsc{pnf} for short) is, by
definition, a series of quantifiers followed by a quantifier-free
formula.  A formula 
\begin{align*}
\exists n_1 \forall m_1 \exists n_2 \forall m_2
\cdots . P[n_1, m_1, n_2, m_2,\ldots]
\end{align*} in \textsc{pnf} is true in
classical logic, if and only if the formula represents a game between
the quantifiers $\exists$ and $\forall$ where the player $\exists$ has a
winning strategy. Every formula is equivalent to a formula in
\textsc{pnf} in classical logic, but it is not the case in $\ha$.  It
may be interesting to think of an extension of $\ha$ from viewpoint of
games which the formulas represent.  We
ask ourselves, ``For which set $\Gamma$ of $L_\ha$-formulas, which
extension $T$ of $\ha$ admits the \emph{prenex normal form theorem}
for $\Gamma$?'' We will syntactically study the question.

For the study, we use \emph{an arithmetical hierarchy of semi-classical principles}, introduced in \cite{akama04:_arith_hierar_of_laws_of}.  In the
hierarchy,  the law of excluded middle and the principle of double
negation elimination are relativized by various formula classes $\Gamma=\Sigma^0_k, \Pi^0_k,\ldots$
$(k\ge0)$. The hierarchy has following axiom schemes:
\begin{align*}(\Gamma\mbox{-}\mathbf{LEM})&\qquad\quad\ \, A \vee  \neg
 A\qquad&(\mbox{$A$ is any $\Gamma$-formula}).\\
(\Gamma\mbox{-}\mathbf{DNE})&\qquad \dneg A \to A\qquad&(\mbox{$A$ is any $\Gamma$-formula}).
\end{align*}
  Any set $X\subseteq\Nset$ in Kleene's arithmetical hierarchy is
identical to $\Nset\setminus(\Nset\setminus X)$. However, not every
formula $A$ is equivalent in $\ha$ to $\dneg A$. So we defined the
\emph{dual} $A^\perp$ of $A$ in a way similar to so-called
\emph{involutive negation of classical logic}. We show that $\ha\vdash
(A^\perp)^\perp \lequiv A$ for any  formula $A$ in \textsc{pnf}, and consider an axiom scheme
\begin{align*}(\Gamma\mbox{-}\mathbf{LEM}')\qquad A \vee  
 A^\perp \qquad(\mbox{$A$ is any $\Gamma$-formula}). 
\end{align*}
The axiom scheme $\slem{k}$ turns out to be equivalent in $\ha$ to
$\slemm{k}$. Motivated by $\Delta^0_k$-sets of Kleene's arithmetical
hierarchy, the hierarchy of semi-classical principles has the following
axiom scheme
\begin{align*}\dlem{k}\qquad (A \lequiv B) \to (A \vee \neg A)\qquad(A\in \Pi^0_{k}, B \in \Sigma^{0}_{k}).
\end{align*}
According to \cite{akama04:_arith_hierar_of_laws_of}, it is weaker than the variant
\begin{align*}\fpdlem{k}\qquad (A \lequiv B) \to (B \vee A^\perp)\qquad(A\in \Pi^0_{k}, B \in \Sigma^{0}_{k}).
\end{align*}

Among these axiom schemes appearing in the arithmetical hierarchy of
semi-classical principles, we answer, ``Which axiom scheme is stronger
than which axiom scheme?''

\begin{theorem} \alabel{thm:main}For any $k\ge0$,
\begin{eqnarray}
&&\slem{k}\ \mbox{proves}\ \plem{k}\ \mbox{in $\ha$}\enspace.\alabel{eq3}\\
&&\dne{k+1}\ \mbox{proves}\ \slem{k}\ \mbox{in $\ha$}\enspace. \alabel{eq4}\\
&&\slem{k}\ \mbox{intuitionistically proves}\ \dne{k}\enspace.\alabel{eq6}\\
&&\plem{k+1}\ \mbox{intuitionistically proves}\
 \slem{k}\enspace.\alabel{eq7}\\
&&\fpdlem{k}\ \mbox{is equivalent in $\ha$ to}\
 \sdne{k}\enspace.\alabel{eq8}\\
&&\sdne{k}\ \mbox{proves $\dlem{k}$ in $\ha$}\enspace.\alabel{eq9}
\end{eqnarray}
\end{theorem}

Let $T$ be a consistent extension of $\ha$. For a formula $A$ of $T$,
 let a formula $A'$ be obtained from $A$ by moving a quantifier of $A$
 over a subformula $D$ of $A$. If the subformula $D$ is \emph{decidable}
 in $T$~(i.e. $T$ proves $D\vee \neg D$), then the formulas $A$ and $A'$
 are equivalent in $T$. Based on this observation, by
 Theorem~\ref{thm:main}, we prove the following:
 
\begin{theorem}[Prenex Normal Form Theorem]\alabel{pnfthm}
  For every $L_\pra$-formula $A$ having at most $k$ quantifiers, we can
 find an $L_\pra$-formula $\hat A$
 in \textsc{pnf} which has $k$ quantifiers and is equivalent in  $\ha+\slem{k}$ to $A$.
\end{theorem}

Actually, for $k$, we can take an ``essential'' number of alternation of
nested quantifiers. See Subsection~\ref{subsec:pnfthm} for detail.

\begin{figure}[t]
\begin{tabular}{ll}
\begin{minipage}{0.5\textwidth}
\unitlength=0.75mm
\begin{picture}(60,96)
\put(-3  ,20  ){\makebox(20,5){$\pdne2$}}
\put(-3  ,60  ){\makebox(20,5){$\pdne3$}}
\put(-3  ,25  ){\makebox(20,5){$\mathsf{fp}\dlem1$}}
\put(-3  ,65  ){\makebox(20,5){$\mathsf{fp}\dlem2$}}

\multiput(23  ,27  )( 0,40){2}{\vector( 1,0){2}}   % fpdlem2 -> sdne2
\multiput(23  ,27  )( 0,40){2}{\vector( -1,0){4}}  % fpdlem2 <- sdne2

\multiput(23  ,22  )( 0,40){2}{\vector( 1,0){2}}   % pdne3   <-
\multiput(23  ,22  )( 0,40){2}{\vector( -1,0){4}}  % pdne3   ->

\put(25  ,25  ){\makebox(20,5){$\sdne1$}}
\put(50  ,30  ){\makebox(20,5){$\plem1$}}
\put(25  ,65  ){\makebox(20,5){$\sdne2$}}
\put(50  ,70  ){\makebox(20,5){$\plem2$}}
\put(47.5, 2.5){\makebox( 0,0){$\slem0=\ha$}}
\put(47.5,12.5){\makebox( 0,0){$\dlem1$}}
\put(47.5,42.5){\makebox( 0,0){$\slem1$}}
\put(0,5.5){\line(1,0){90}}
\put(0,45){\line(1,0){90}}
\put(47.5,52.5){\makebox( 0,0){$\dlem2$}}
\put(0,85){\line(1,0){90}}
\put(47.5,82.5){\makebox( 0,0){$\slem2$}}
\put(50  ,20  ){\makebox(20,5){$\sllpo1$}}
\put(77 ,20  ){\makebox(20,5){$\bsdne2$}}
\put(82 ,15  ){\makebox(20,5){$\pllpo1$}}
\put(50  ,60  ){\makebox(20,5){$\sllpo2$}}
\put(77 ,60  ){\makebox(20,5){$\bsdne3$}}
\put(82 ,55  ){\makebox(20,5){$\pllpo2$}}
\multiput(47.5,87  )(0,1.5){6}{\makebox( 0,0){$\cdot$}}
\put(47.5, 97  ){\makebox( 0,0){$\pa$}}
\multiput(40  ,25  )( 0,40){2}{\vector( 1,-2){5}}   % sdne2 \ dlem2,  sdne1 \ dlem1

\multiput(55  ,20  )( 0,40){2}{\vector(-1,-1){5}}   % sllpo2-> dlem2

\multiput(45  ,40  )( 0,40){2}{\vector(-1,-2){5}}   % slem2 -> sdne2

\multiput(50  ,40  )( 0,40){2}{\vector( 1,-1){5}}   % slem2 -> plem2

\multiput(47.5,10  )( 0,40){2}{\vector( 0,-1){4}}

\put(55  ,30  ){\vector( 0,-1){5}}                % plem1 -> slem1

\put(55  ,70  ){\vector( 0,-1){5}}                % plem2 -> sllpo2

\multiput(72  ,22.5)( 0,40){2}{\vector( 1, 0){4}} % sllpo{2} -> bsdne{3}, sllpo{1} -> bsdne{2}
\multiput(75 ,22.5)( 0,40){2}{\vector(-1, 0){4}}  % sllpo{2} <- bsdne{3}, sllpo{1} <- bsdne{2}
\multiput(72  ,17.5)( 0,40){2}{\vector( 1, 0){4}} %  -> pllpo{N+1}
\multiput(75 ,17.5)( 0,40){2}{\vector(-1, 0){4}} %  <- pllpo{N+1}
\end{picture}
\end{minipage} &

\begin{minipage}{0.5\textwidth}
\Large \begin{align*}
\xymatrix{
\vdots & \plim^\omega \A\\ 
\plim^3 \A \ar[u]^{\iota_{\plim^3 \A}} \ar[ru]^\vdots\\
\plim^2 \A \ar[u]^{\iota_{\plim^2 \A}} \ar[ruu]\\
\plim   \A \ar[u]^{\iota_{\plim \A}} \ar[ruuu]\\
        \A \ar[u]^{\iota_\A} \ar[ruuuu]}
\end{align*}
\end{minipage}
\end{tabular}
\alabel{thehierarchy}
\caption{The left is the arithmetical hierarchy of semi-classical
 principles. The one-way
 arrows means implication which is not reversible. The non-reversibility, the 
the axiom schemes
principle $\sllpo{k}$, $\bsdne{k+1}$ and $\pllpo{k}$  are not discussed in
 this paper, but in 
 \cite{akama04:_arith_hierar_of_laws_of}. 
The right
diagram consisting of \textsc{pca}s and 
 homomorphisms is a 
 \emph{colimit diagram}, in the category of
 \textsc{pca}s and homomorphisms between them. The vertical arrows are
 canonical injections~(see
 Section~\ref{sec:limitingPCA} for detail)
\alabel{Bhierarchy}}
\end{figure}

\subsection{Iterated Limiting PCA and Realizability Interpretations}

\cite{MR2030297} introduced a limit operation $\plim(\bullet)$ for partial
combinatory algebras~(\textsc{pca}s for short) such that  from any \textsc{pca}
$\A$, the limit operation $\plim(\bullet)$
 builds  hierarchies $\{\plim^\alpha \A\}_{\alpha=0,1,\ldots,\omega}$ of
 \textsc{pca}s satisfying Figure~\ref{Bhierarchy}~(right).
 The limit operation corresponds to the jump operation of
the arithmetical hierarchies, as in Shoenfield's limit lemma~(see
\cite{MR982269} for instance). The introduction of the limit operation
aimed to represent approximation algorithms needed in proof
animation~\citep{Hayashi01:_towar_animat_proof}. Hayashi proposed proof
animation in order to make interactive formal proof development
easier.

In this paper, we provide a realizability interpretation of  $\pa$ by a \textsc{pca}
$\plim^\omega \A$ for every \textsc{pca} $\A$. 

\begin{theorem}[Iterated Limiting Realizability Interpretation]\alabel{thm:realPA}For any \textsc{pca} $\A$ and for any nonnegative integer
 $k$, the system $\ha+\sdne{k+1}$ is sound by the realizability
 interpretation for the \textsc{pca} $\plim^k(\A)$. 
 $\pa$ is sound by the realizability interpretation for the \textsc{pca}
 $\plim^\omega(\A)$.
\end{theorem} 
Let us call realizability
interpretation by a \textsc{pca} $\plim^\alpha\A$ an \emph{iterated
limiting realizability interpretation} $(\alpha=0,1,2,\ldots,\omega)$.
The feature of our
realizability interpretation of $\pa$ are:
\begin{itemize}\item
if non-constructive objects are
allowed to exist by the double negation elimination axioms, the
realization of the non-constructive objects requires 
the jump
of mathematical intuition. The jump is achieved by the limit.

\item Our realizability interpretation of $\pa$ is simpler than those by
\cite{berardi98computational} and \cite{avigad00}. They embedded
classical logic to intuitionistic logic by the G\"odel-Gentzen's
negative translation~(see Section~81 of \cite{MR0051790} for example) or
the Friedman-Dragalin translation, and then carried out the recursive
realizability interpretation. However, they needed a special observation
in interpreting the translation results of logical principles.
\cite{MR2121494} developed a theory for ``classical logic as limit.''
\end{itemize}

\subsection{Two Consequences of Our Prenex Normal Form Theorem and Our Iterated Limiting
Realizability Interpretation of $\pa$} We derive a result for
\emph{independence-of-premise schemes}~(see Section~1.11.6 of
\cite{MR48:3699}), and that for $n$-consistent extension of $\ha$.

\begin{definition}[Independence-of-premise scheme]  Let $\Gamma$ be a set of $L_\pra$-formulas.
$(\Gamma\mbox{-}\textbf{IP})$ is an axiom scheme
\begin{align*}
(A\to \exists m.\, B)\to \exists m.\, (A\to B)\end{align*}
where $m$ does not
 occur free in $A$, $A$ is any in $\Gamma$, and $B$ is any $L_\pra$-formula.
\end{definition}

Let an $F_n$-formula be any  $L_\pra$-formulas having at most $n$ quantifiers.

\begin{theorem}[Non-derivability between $\fip{k+1}$ and $\sdne{k+1}$]\alabel{thm:fip} $\ha +\sdne{k + 1} +
\fip{k + 1}$ does not admit a realizability interpretation by the
 \textsc{pca} $\plim^k(\Nset)$, where $\Nset$ is the \textsc{pca} of all
 natural numbers such that the partial application operation $\{n\}(m)$
 is the application of the unary partial recursive function of G\"odel
 number $n$ applied to $m$.
 Hence $\sdne{k + 1}\not \vdash_\ha \fip{k + 1}$ and $ \fip{k + 1}\not \vdash_\ha\sdne{k + 1}$.
\end{theorem}

No reasonable subsystem $T$ of $\ha$ seems to admit prenex normal form
theorem, because for all $k$, $T$ does not prove $\fip{k}$.

\bigskip
The next consequence of  our prenex normal form theorem~(Theorem~\ref{pnfthm}) and our iterated limiting
realizability interpretation~(Theorem~\ref{thm:realPA}) of $\pa$
is about
 ``$\pa$ is unbounded extension of $\ha$.''  

Before \cite{akama04:_arith_hierar_of_laws_of}, strict infinite
hierarchies of formal arithmetics $\ha\subsetneq T_1 \subsetneq T_2
\subsetneq\cdots \subsetneq \pa$ was provided in a proof of a theorem
``any set $\Gamma$ of $L_\pra$-sentences with bounded
quantifier-complexity does not axiomatize $\pa$ over $\ha$.''  The proof
was sketched in Section~3.2.32 of \cite{MR48:3699}, and was based on
C.~Smory{\'n}ski's idea given in his unpublished note ``\emph{Peano's
arithmetic is unbounded extension of Heyting's arithmetic}.''  
\cite{MR48:3699}
used a \emph{realizability interpretation}~(\cite{MR0015346}) but the
realizers are G\"odel numbers of partial functions \emph{recursive in a
complete $\Pi^0_k$-set} of the \emph{Kleene's arithmetical hierarchies}.

We say an arithmetic $T$ is \emph{$n$-consistent}, provided every
$\Sigma^0_n$-sentence provable in $T$ is true in the standard model
$\omega$. Note that $\ha$ is $n$-consistent for each positive integer $n$.

\begin{theorem}[PA as bounded extension of HA]\alabel{thm:smotroaka}  Let $n\ge 2$ be a natural number, and
$\Gamma$ be a set of $L_\pra$-sentences containing at most $n$ quantifiers.
If $\ha+\Gamma$ is $n$-consistent, then $\ha+\Gamma$  does not prove the
 axiom scheme $\slem{n+1}$.
\end{theorem}

The background and a possible research direction of the theorem is given in
Section~\ref{sec:classical}.  The rest of the paper is organized as
follows.  In Section~\ref{sec:hierarchy}, the hierarchies of logical
systems between $\ha$ and $\pa$ are introduced to discuss the prenex
normal form theorem~(Theorem~\ref{pnfthm}). In Section~\ref{sec:limitingPCA}, we introduce 
iterated autonomous limiting \textsc{pca}s, In
Section~\ref{sec:classical}, by using the such \textsc{pca}s, we
introduce and study the iterated limiting realizability interpretation
of arithmetics between $\ha$ and $\pa$.  In Subsection~\ref{subsec:proofs}, we verify
Theorem~\ref{thm:fip} and Theorem~\ref{thm:smotroaka}.

\section{Hierarchy of Semi-classical Principles}
\alabel{sec:hierarchy}

When we move quantifiers of a formula $A$ outside the scope of
propositional connectives, we ask ourselves when the resulting formula
$A'$ is equivalent in $\ha$ to the formula $A$.

\begin{lemma}\alabel{lem:mqdc}If a variable $n$ does not occur in a formula
 $A$, then intuitionistic predicate logic
  $\IQC$ proves:
(1) \newcounter{firstmqdcassert}\setcounter{firstmqdcassert}{1}$A\vee\forall n B\to   \forall n(A\vee B)$;
(2)\newcounter{secondmqdcassert}\setcounter{secondmqdcassert}{2}
	$\exists n(A\circ B)\lequiv A\circ\exists n B$ for
	$\circ=\vee,\wedge$; and
(3)\newcounter{thirdmqdcassert}\setcounter{thirdmqdcassert}{3} $\forall n(A\wedge B)\lequiv A\wedge\forall n B$.
\end{lemma}

As usual, the symbol $\vdash$ denotes the derivability.

\begin{fact}\alabel{dec}  Suppose $T$ is a formal system of arithmetic  extending $\IQC$. We say
 a formula $D$ of $T$ is  \emph{decidable} in $T$,
 if $T\vdash D\vee\neg D$.
  \begin{enumerate}
   \item If formulas $D$ and $D'$ are decidable in $T$, so are $ \neg D$ and
    $D\circ D'$ for $\circ=\wedge,\vee,\to$.

   \item \alabel{assert:bq} If a formula $D$ is decidable in $\ha$, then bounded universal quantifications $\forall
	n<t.\ D$ and $\exists n < t.\ D$ are decidable in $\ha$. 

   \item \alabel{assert:qf} Every $\Sigma^0_0$-formulas is decidable in $\ha$.
  \end{enumerate}
\end{fact}

\begin{fact}\alabel{lem:int}None of the following two formulas
 \eqref{assert:decimp} and \eqref{assert:uniin} are
 provable in $\IQC$ but both of two formulas $(D\vee \neg D)\to
 \eqref{assert:decimp}$ and  $(D'\vee\neg D')\to
 \eqref{assert:uniin}$ are.
 \begin{align}
 \alabel{assert:decimp}
& (D\to B) \lequiv(\neg D\vee B).\\
 \alabel{assert:uniin}&\forall n(D'\vee B)\to D'\vee\forall n
  B\qquad\mbox{($n$ does not occur free in $D'$).} 
\end{align}
\end{fact}
$\IQC$ with the scheme \eqref{assert:uniin}
added is complete for the class of Kripke models of constant domains, and $\ha$ plus
the schema is just $\pa$,  as explained in Section~1.11.3 of \cite{MR48:3699}.

\subsection{Proof of Theorem~\ref{thm:main}}
\alabel{sec:3.2}

For a formula $A$, we define a formula $A^\perp$ classically equivalent to $\neg
A$, as follows:

\begin{definition}
  For any formula $A$, we define
 the \emph{dual} $A^\perp$ as follows:
 \begin{itemize}
 \item When $A$ is prime, $A^\perp$ is the negation $\neg A$.
 \item When $A$ is a negated formula $\neg B$, then $A^\perp$ is $B$.
 \item When $A$ is $B\vee C$, then $A^\perp$ is $B^\perp\wedge C^\perp$.
 \item When $A$ is $B\wedge C$, then $A^\perp$ is $B^\perp\vee C^\perp$.
 \item When $A$ is $B\to C$, then $A^\perp$ is $B\wedge C^\perp$.
 \item When $A$ is $\forall n.\, B$, then $A^\perp$ is $\exists n.\,
   B^\perp$.
 \item When $A$ is $\exists n.\, B$, then $A^\perp$ is $\forall n.\, B^\perp$.
 \end{itemize}
\end{definition}
The dual operation is more manageable than the propositional connective
$\neg$.
\begin{fact}\alabel{fact:aho}
\begin{enumerate}\item
$\ha\vdash \dual P\lequiv \neg P\quad  \qquad(\mbox{$P$ is a
		      $\Sigma^0_0$-formula.})$ \alabel{0dec}
\item $\ha\vdash (A^\perp)^\perp\lequiv  A\  \qquad(\mbox{$A$ is a
		      $\Sigma^0_k$-formula or a
		      $\Pi^0_k$-formula.})$\alabel{1dec}
\end{enumerate}
\end{fact}
\begin{proof}\eqref{0dec} By induction on $P$. 
\eqref{1dec}  First consider the case the formula  $A$ is a $\Sigma^0_k$-formula. Then
 $A$ is written as $\exists
 n_1 \forall n_2 \exists n_3\cdots Q n_k.\ P$ for some
 $\Sigma^0_0$-formula $P$. Then
$(A^\perp)^\perp$ is $\exists
 n_1 \forall n_2 \exists n_3\cdots Q n_k.\ (P^\perp)^\perp$. 
The Assertion~\eqref{0dec} implies $\vdash_\ha (P^\perp)^\perp \lequiv \dneg
 P$.  But Fact~\ref{dec}~\eqref{assert:qf}, implies the decidability of
 $P$. So $\vdash_\ha \dneg P \lequiv P $. Hence $\vdash_\ha
 (P^\perp)^\perp\lequiv P$.
Therefore $\vdash_\ha (A^\perp)^\perp \lequiv A$. When $A$ is a
 $\Pi^0_k$-formula, the proof is similar.
\end{proof}

The axiom scheme  $\slemm{k}$ is the axiom scheme consisting of the following form:
\begin{eqnarray}
&     &\exists n_1\forall n_2\cdots Q  n_{k-1} \Q  n_k P[n_1,\ldots, n_k]\nonumber\\
&\vee&\forall m_1\exists m_2\cdots \Q  m_{k-1} Q  m_k \left(P [m_1,
						       \ldots, m_k] \right)^\perp.\alabel{slemk}
\end{eqnarray}
Here $P[n_1, \ldots, n_k]$ and $P[m_1, \ldots, m_k]$ are $\Sigma^0_0$-formulas
possibly containing  free variables other than indicated variables, and the quantifier $Q$ is $\forall$ for odd
$k$ and is $\exists$ otherwise.  $\Q$ is $\exists$ if $Q$ is $\forall$,
and is $\forall$ otherwise.
\begin{align}\slemm{k}\vdash_\ha \plemm{k}\ \mbox{and}\  \plemm{k}\vdash_\ha
		      \slemm{k} \alabel{2dec}
\end{align}
 follows from  Fact~\ref{fact:aho}~\eqref{1dec}, because the dual of a
 $\Sigma^0_k$-formula~($\Pi^0_k$-formula, resp.) is a $\Pi^0_k$-formula~($\Sigma^0_k$-formula, resp).

\begin{fact}\alabel{fac} For any formula $A$, $\IQC$ proves
  \newcounter{assertdualbot}
\setcounter{assertdualbot}{1} (\theassertdualbot)
 $\neg(A\wedge A^\perp)$ and 
\newcounter{assertdualneg}\setcounter{assertdualneg}{2}
(\theassertdualneg) $(A\vee A^\perp)\to (A^\perp\lequiv   \neg A)$.
\end{fact}
\begin{proof}
(\theassertdualbot) The proof is by induction on the structure of
  $A$. When $A$ is prime or negated, the assertion is trivial. When $A$
  is $B\vee C$, let us assume $B\vee C$ and the dual $A^\perp$, that is,
  $B^\perp\wedge C^\perp$. The first conjunct contradicts by the
  induction hypothesis in case of $B$, and the second by the induction
  hypothesis in case of $C$. So, $\neg(A\wedge A^\perp)$. When $A$ is
  a conjunction, the assertion is similarly verified. When $A$ is $B\to
  C$, let us assume $B\to C$ and the dual, that is $B\wedge
  C^\perp$. From the first conjunct $B$ and $B\to C$, we infer $C$,
  which contradicts by the induction hypothesis against the second conjunct
  $C^\perp$. When $A$ is $\forall n.\, B [n]$, let us assume $\forall
  n.\ B[n]$ and the dual $\exists n.\, (B [n])^\perp$. For a fresh
  variable $m$, assume $(B[m])^\perp$. But we can infer $B [m]$ from
  $A$. This contradicts against the induction hypothesis. When $A$ is
  existentially quantified, the assertion is similarly verified.
(\theassertdualneg)  The Assertion~(\theassertdualbot) implies $(A\vee \dual A ) \to (\dual A\to \neg
       A)$, while $(A\vee \dual A ) \to (\neg A\to \dual
       A)$ is immediate.
\end{proof}

The two axiom schemes $\slemm{k}$ and $\slem{k}$ are
equivalent over $\ha$, as we prove below:

\begin{lemma}\alabel{lem3.8} For any $k\ge0$, 
(1)  $\slemm{k}\vdash_\IQC\slem{k}$, and (2)  $\slem{k}\vdash_\ha\slemm{k}$.

\end{lemma}
\begin{proof}The first assertion follows from Fact~\ref{fac}~(\theassertdualneg) in
 $\IQC$. 
The second assertion is proved
by  induction on $k$. The assertion holds
 for $k=0$, because  $\vdash_\ha \slemm{0}$ follows from
 Fact~\ref{dec}~\eqref{assert:qf} and Fact~\ref{fact:aho}~\eqref{0dec}. Let $k>0$. Consider
a $\Sigma^0_k$-formula $\exists n.\, B$ with $B$ being any
$\Pi^0_{k-1}$-formula. 
 By the induction hypothesis,  we have $\slem{k}\vdash_\ha
 \slemm{k-1}$. Because $\slemm{k-1}$ and $\plemm{k-1}$ are equivalent over
 $\ha$ by  \eqref{2dec}, we have
 $\slem{k}\vdash_\ha\dual{B} \vee B$. 
By this and Fact~\ref{fac}~(\theassertdualneg), we have
 $\slem{k}\vdash_\ha B^\perp\lequiv \neg B$. So  $\slem{k}\vdash_\IQC
 \exists n. B\vee \forall n. \neg B$ implies $\slem{k}\vdash_\ha
 \exists n. B\vee \forall n. \dual B$. Therefore $\slem{k}\vdash_\ha
 \slemm{k}$.  
\end{proof}

We prepare  the proof of Theorem~\ref{thm:main}~\eqref{eq4} below.
An instance \eqref{slemk} of $\slemm{k}$ is equivalent in $\pa$ to the
following $\Sigma^0_{k+1}$-formula:
\begin{align}
    \alabel{eq:se}
    \exists n_1(\forall m_1\forall n_2) (\exists m_2\exists n_3)\cdots (Q
    m_{k-2}Q  n_{k-1}) (\Q  m_{k-1}\Q  n_{k}) Q  m_k \nonumber \\
\left(P[n_1,\ldots,
 n_k]\;\vee\;\neg P [m_1, \ldots, m_k]\right).
 \end{align}
Here $P[n_1,\ldots, n_k]$ and $\neg P[m_1, \ldots, m_k]$ are
$\Sigma^0_0$-formulas possibly containing free variables other than
indicated variables.

We apply $\sdne{k+1}$ to the \emph{G\"odel-Gentzen translation}~(Section~81 of \cite{MR0051790}) result of \eqref{eq:se}.
\begin{lemma}\alabel{lem:b}
  Let $k\geq 1$. The $\Sigma^0_{k+1}$-formula \eqref{eq:se} is provable in $\ha+\dne{k+1}$.
  \end{lemma}
\begin{proof} It is easy to see that the $\Sigma^0_{k+1}$-formula
 \eqref{eq:se} is
 equivalent in a classical logic to an instance of  $\slemm{k}$. So,
$\ha$ proves
 the G\"odel-Gentzen translation of
  \eqref{eq:se}, which is
  obtained from \eqref{eq:se}
  \begin{enumerate}
  \item by replacing each $(\exists l)$ with $(\neg\forall l\neg)$; and 
  \item by replacing the disjunction $P[n_1,\ldots,n_k]\ \vee\ \neg
	P[m_1,\ldots,m_k]$ with a formula $\neg(\neg
    P[n_1,\ldots,n_k]\ \wedge\ \dneg P[m_1,\ldots,m_k])$.
  \end{enumerate}
  However,
  \begin{enumerate}
  \item for each formula $A$,  $\IQC\vdash \neg\forall l\neg A \lequiv \dneg \exists l.\, A$; and
  \item  $\ha \vdash P[n_1,\ldots,n_k] \vee \neg P[m_1,\ldots,m_k]\lequiv \neg(\neg
    P[n_1,\ldots,n_k]\wedge \dneg P[m_1,\ldots,m_k])$, by Fact~\ref{dec}~\eqref{assert:qf}.
  \end{enumerate}
  So, $\ha$ proves a formula obtained from \eqref{eq:se} by only inserting
  $\dneg$ just before
  each existential quantifier. The resulting formula is 
  \begin{displaymath}
    \dneg\exists n_1(\forall m_1\forall n_2) (\dneg\exists m_2\dneg\exists
    n_3)\cdots  \left(P[n_1,\ldots, n_k]\;\vee\;\neg P[m_1,\ldots, m_k]\,\right),\quad (f_0)
  \end{displaymath}
  and ends with
  \begin{description}
  \item[($o_0$)]\ \  $\forall n_{k-1}\dneg\exists m_{k-1}\dneg\exists n_{k}\forall 
     m_{k} \left( P[\vec{n}]\vee\neg P[\vec{m}]\right)$ for odd $k$; and
   \item[($e_0$)]\ \  $\forall n_{k-1}\dneg\exists m_{k} \left( P[\vec{n}]\vee\neg
     P[\vec{m}]\right)$ for even $k$.
  \end{description}
  In each case, the rightmost $\dneg$ is just before a $\Sigma^0_{1+(k
    \bmod 2)}$-formula.  So, if we can use $\dne{1+(k \bmod 2)}$, then
    the rightmost $\dneg$('s) in the subformulas ($o_0, e_0$) can be
    safely eliminated from the formula $(f_0)$. But $\dne{1+(k \bmod
    2)}$ follows from $\dne{k+1}$. Thus $\dne{k+1}$ proves in $\ha$ the
    formula ($f_0$) with the rightmost $\dneg$('s) eliminated from the
    end-part $(o_0, e_0)$. The resulting formula ($f_1$) ends with
  \begin{description}
  \item[($o_1$)]\ \  $\dneg\exists m_{k-3}\dneg\exists n_{k-2}(\forall m_{k-2}\forall n_{k-1})(\exists m_{k-1}\exists n_k )\forall m_k  \left( P[\vec{n}]\vee\neg P[\vec{m}]\right)$ for odd  $k$; and
  \item[($e_1$)]\ \  $\dneg\exists m_{k-2}\dneg\exists n_{k-2}(\forall m_{k-1}\forall n_k )(\exists m_k ) \left( P[\vec{n}]\vee\neg P[\vec{m}]\right)$ for even $k$.
  \end{description}
  In each case, the rightmost $\dneg$ is just before a $\Sigma^0_{3+(k \bmod
    2)}$-formula.  So, if we can use $\dne{3+(k \bmod
    2)}$, then the rightmost $\dneg$'s in ($o_1, e_1$) can be safely
    eliminated from $(f_1)$. But $\dne{3+(k \bmod
    2)}$ follows from $\dne{k+1}$. Thus $\dne{k+1}$ proves in $\ha$ the
  formula ($f_1$) with the rightmost $\dneg$'s eliminated from the
  end-part $(o_1, e_1)$. 

  By iterating this argument, we can safely eliminate all $\dneg$'s from
  $(f_0)$. This establishes that $\dne{k+1}$ proves in $\ha$ the
  $\Sigma^0_{k+1}$-formula \eqref{eq:se}. This completes the proof of Lemma~\ref{lem:b}.
\end{proof}

\bigskip
We will present the proof of Theorem~\ref{thm:main}.

\medskip
 Assertion~\eqref{eq3} ``$\slem{k}\vdash_\ha \plem{k}$'' is verified as follows: By Lemma~\ref{lem3.8}, 
we see that for every $\Sigma^0_0$-formula $P[\vec{n}]$, a disjunction of $\exists n_1\forall n_2\cdots Q  n_k\neg
  P[\vec{n}]$ and $\forall n_1\exists n_2\cdots \Q  n_k P[\vec{n}]$ is
 deducible in $\ha$ from $\slem{k}$. When the first disjunct holds, then  it 
  contradicts against the dual of the first disjunct by Fact~\ref{fac}~(\theassertdualbot), and thus we have the negation
  $\neg\forall n_1\exists n_2\cdots \Q  n_k P[\vec{n}]$ of the dual. In the
  other case, then we have the second disjunct $\forall n_1\exists n_2\cdots \Q  n_k
  P[\vec{n}]$. In both cases, we have $\forall n_1\exists n_2\cdots \Q  n_k
  P[\vec{n}]\;\vee\;\neg\forall n_1\exists n_2\cdots \Q  n_k P[\vec{n}]$,
 which is an instance of
  $\plem{k}$.

\bigskip
Assertion~\eqref{eq4} ``$\sdne{k+1}\vdash_\ha \slem{k}$'' of  Theorem~\ref{thm:main} will be proved by induction on $k$.
 The  case $k=0$ follows from Fact~\ref{dec}~\eqref{assert:qf}.
Next consider the   case $k>0$.

\begin{claim}\alabel{claim:I}
Suppose that $j\le k$ is a positive odd number and
 that a variable $m_j$ does not occur free in a $\Pi^0_{k-j}$-formula $\forall n_{j+1} \exists n_{j+2}\cdots Q n_k .\ P[n_1,\cdots, n_k]$. Then
 $\ha+\sdne{k+1}$ proves the following equivalence formula:
\begin{align*}
&\forall m_j \left(
			  \forall n_{j+1} \exists n_{j+2}\cdots Q n_k .\ P[n_1,\cdots, n_k]
\  \vee   \qquad  \exists m_{j+1} \forall m_{j+2}\cdots
			  \overline{Q} m_k .\ \neg P[m_1, \ldots, m_k]\right)\\
&\lequiv \ \ \left(
			  \forall n_{j+1} \exists n_{j+2}\cdots Q n_k . P[n_1,\cdots, n_k]
        \; \vee\;   \forall m_j  \exists m_{j+1} \forall m_{j+2}\cdots
			  \overline{Q} m_k .\neg P[m_1, \ldots, m_k]\right).
\end{align*}
\end{claim}
\begin{proof}In the left-hand side of the equivalence formula,
 we can easily see the first disjunct $\forall n_{j+1} \exists n_{j+2}
 \cdots Q n_k.\ P[n_1, \ldots, n_k]$ is a $\Pi^0_{k-j}$-formula.  The
 system $\ha+\sdne{k+1}$ proves $\sdne{k-j+1}$ which proves $\slem{k-j}$
 by the induction hypothesis on Assertion~\eqref{eq4} of
 Theorem~\ref{thm:main}. Hence the system $\ha+\sdne{k+1}$ proves
 $\plem{k-j}$ by Assertion~\eqref{eq3} of Theorem~\ref{thm:main}. Thus
 the $\Pi^0_{k-j}$-disjunct $\forall n_{j+1} \exists n_{j+2} \cdots Q
 n_k.\ P [n_1, \ldots, n_k]$ of the left-hand side is decidable in
 $\ha+\sdne{k+1}$, where the variable $m_j$ does not occur free. Because
 of Lemma~\ref{lem:mqdc} and Fact~\ref{lem:int},
 the left-hand side and the right-hand side of the equivalence formula
 is indeed equivalent in the system $\ha+\sdne{k+1}$.
\end{proof}

Next, we will consider when the universal
quantifier can be safely moved over $\Sigma^0_{k-i+1}$-disjunct where $i\ge1$.

\begin{claim}\alabel{claim:II}
Suppose that $i\le k$ is a positive even number and
 that a variable $n_i$ does not occur free in a  $\Sigma^0_{k-i+1}$-disjunct $\exists m_i \forall m_{i+1}  \cdots
 \overline{Q} m_k.\ \neg P [m_1, \ldots, m_k]$ does not contain a free
 variable $n_i$. Then
 $\ha+\sdne{k+1}$ proves the following equivalence formula
\begin{align*}
  &\ \ \ \ \ \ \ \forall n_i \left(
			  \exists n_{i+1} \forall n_{i+2}\cdots Q n_k .\ P [n_1, \ldots, n_k]\  \vee\  \exists m_i \forall m_{i+1}\cdots
			  \overline{Q} m_k .\ \neg P[m_1, \ldots, m_k]\right)\\
&\lequiv \ \ \left(\forall n_i \;\ \exists n_{i+1}\forall n_{i+2}\cdots Q n_k .\ P [n_1, \ldots, n_k]
        \ \vee  \  \exists m_i  \forall m_{i+1} \cdots
			  \overline{Q} m_k .\ \neg P[m_1, \ldots, m_k]\right).
\end{align*} 

\end{claim}
\begin{proof}In the left-hand side of the equivalence formula,
 we see that the second disjunct $\exists
 m_i \forall m_{i+1} \cdots \overline{Q} m_k.\ \neg P [m_1, \ldots, m_k]$ is
 a $\Sigma^0_{k-i+1}$-formula. It is decidable in $\ha+\sdne{k+1}$,
 because $\sdne{k+1}$ proves $\sdne{k-i+2}$ which proves $\slem{k-i+1}$
 by the induction hypothesis of Assertion~\eqref{eq4} of Theorem~\ref{thm:main}. 
 The decidable
 $\Sigma^0_{k-i+1}$-disjunct $\exists m_i \forall m_{i+1}  \cdots
 \overline{Q} m_k.\ \neg P [m_1, \ldots, m_k]$ does not contain a free
 variable $n_i$. So move the universal quantifier $\forall n_i$ over the decidable
 $\Sigma^0_{k-i+1}$-disjunct.  The resulting formula is the right-hand side
 of the equivalence formula. It is equivalent in $\ha+\sdne{k+1}$ to the left-hand side of
 the equivalence formula, by Lemma~\ref{lem:mqdc} and Fact~\ref{lem:int}.
\end{proof} 

We continue the proof of Assertion~\eqref{eq4} ``$\dne{k+1}\ \vdash_\ha \slem{k}$'' of Theorem~\ref{thm:main}.
To an instance~\eqref{slemk} of $\slemm{k}$, apply
Lemma~\ref{lem:mqdc}, Fact~\ref{lem:int}, Claim~\ref{claim:I} with
 $j=1$, and Claim~\ref{claim:II} with $i=2$. Next apply 
Lemma~\ref{lem:mqdc}, Fact~\ref{lem:int}, Claim~\ref{claim:I} with
$j=3$, and
 Claim~\ref{claim:II} with $i=4$. Then repeatedly apply them with $(i,j)=(5,6),(7,8),\ldots$.
 $\ldots$, in this order. Then a formula \eqref{slemk} is
 equivalent in $\ha+\sdne{k+1}$ to the $\Sigma^0_{k+1}$-formula~\eqref{eq:se}. But the
 formula \eqref{eq:se} is
 provable in  $\ha+\sdne{k+1}$  by Lemma~\ref{lem:b}. Hence every instance
\eqref{slemk} of $\slemm{k}$ is  provable in the system
 $\ha+\sdne{k+1}$. Thus the system $\ha+\sdne{k+1}$ proves $\slemm{k}$
 and thus $\slem{k}$ by Lemma~\ref{lem3.8}. This completes the proof of Assertion~\eqref{eq4}.

\bigskip  To prove  Assertion~\eqref{eq6} ``$\slem{k}\vdash_\IQC
 \sdne{k}$,'' let us assume $\dneg A$ with $A$ being a
 $\Sigma^0_k$-formula.
  By $\slem{k}$, we have $A\vee\neg A$. In case of $\neg A$, by the
  assumption $\dneg A$, we have contradiction, from which $A$
 follows. Hence we concludes $\dneg A\to
  A$.

\bigskip  To prove  Assertion~\eqref{eq7} ``$\plem{k+1}\vdash_\IQC
  \slem{k}$,'', note that any
  $\Sigma^0_k$-formula $B$ is  equivalent in $\IQC$ to a $\Pi^{0}_{k+1}$-formula $\forall n.\,
  B$ where the variable $n$ is
 fresh.  Because $\ha+\plem{k+1}$ proves
$\forall n.\,  B\;\vee \neg\forall n.\, B$, so does $B\vee\neg B$, an
 instance of  $\slem{k}$.

\bigskip We will prove Assertion~\eqref{eq8} ``$\sdne{k}$ is equivalent
in $\ha$ to $\fpdlem{k}$'' of Theorem~\ref{thm:main}. First we will
prove ``$\sdne{k}\vdash_\ha \fpdlem{k}$.'' Let us assume $\dne{k}$. Let
$P[n_1,\ldots, n_k]$ and $R[m_1, \ldots, m_k]$ be $\Sigma^0_0$-formulas
possibly containing free variables other than indicated variables.  Also
assume the following equivalence formula between a $\Sigma^0_k$-formula
and a $\Pi^0_k$-formula:
\begin{align}
&\exists n_1\forall n_2\cdots Q  n_{k-1} \Q  n_k.\, P[n_1,\ldots, n_k]\ \
 \lequiv\ \ \forall m_1\exists m_2\cdots \Q  m_{k-1} Q  m_k.\, R[m_1, \ldots, m_k],    \alabel{equiv:4}
\end{align}
We will derive the following disjunction of two $\Sigma^0_k$-formulas:
\begin{align}
\exists n_1\forall n_2\cdots Q  n_{k-1} \Q  n_k.\, P[n_1,\ldots, n_k]\
  \vee\ \ \exists m_1\forall m_2\cdots Q  m_{k-1} \Q  m_k.\, \left( R
[ m_1, \ldots, m_k]\right)^\perp.   \alabel{eq:5}
\end{align}

\begin{claim}
The disjunction \eqref{eq:5} is equivalent in $\ha+\dne{k}$ to a 
$\Sigma^0_k$-formula:
\begin{align}
  \alabel{eq:6}
 \exists n_1 \exists m_1\forall n_2\forall m_2\cdots \Q  n_k \Q  m_k (P[n_1, \ldots, n_k] \;\vee\;\neg R[m_1, \ldots, m_k]).
\end{align}
\end{claim}
\begin{proof}The claim is proved by Lemma~\ref{lem:mqdc}, in a similar argument as
 the Assertion~\eqref{eq4} of Theorem~\ref{thm:main} is. Since the
 $\Sigma^0_k$-formula \eqref{eq:6} is obtained from the
 disjunction~\eqref{eq:5} by moving the quantifiers $\exists
 n_{2i-1},\exists m_{2i-1},\forall n_{2i},\forall m_{2i}\
 (i=1,2,\ldots)$ out of the scope of the disjunction, the equivalence
 between (\ref{eq:5}) and (\ref{eq:6}) in $\ha+\dne{k}$ is established
 by showing that the movement of the quantifiers are safe. The
 existential quantifiers $\exists n_{2i-1},\exists m_{2i-1}$ are safely
 moved by Lemma~\ref{lem:mqdc}. Each quantifier $\forall n_{2i}$ is
 moved over a $\Pi^0_{k-2i+1}$-disjunct $\forall m_{2i}\exists
 m_{2i+1}\cdots \Q m_k \neg R$, and each quantifier $\forall m_{2i}$
 over a $\Sigma^0_{k-2i}$-disjunct $\exists n_{2i+1}\forall
 n_{2i+2}\cdots \Q m_k P$. Here the $\Pi^0_{k-2i+1}$-disjunct and the
 $\Sigma^0_{k-2i}$-disjunct are both decidable by
 Theorem~\ref{thm:main}. So each $\forall n_{2i}$ and $\forall m_{2i}$
 are safely moved. This completes the verification of the claim.\end{proof}

\medskip To complete the verification of Assertion~\eqref{eq8}
``$\sdne{k}\vdash_\ha \fpdlem{k}$,'' it is sufficient to show that the
$\Sigma^0_k$-formula (\ref{eq:6}) from the equivalence
formula~\eqref{equiv:4}, by using $\sdne{k}$.

In view of $\dne{k}$, we have only to derive the double negation of the
$\Sigma^0_k$-formula~\eqref{eq:6}. So assume the negation of the
$\Sigma^0_k$-formula~\eqref{eq:6}, that is,
\begin{displaymath}
\neg\exists n_1 \exists m_1\forall n_2\forall m_2\cdots \Q  n_k \Q  m_k (P[n_1,\ldots, n_k] \;\vee\;\left( R[m_1, \ldots, m_k])^\perp\right).
\end{displaymath}
It is equivalent in $\ha+\dne{k}$ to the dual
\begin{align}
  \alabel{eq:7}
  \forall n_1 \forall m_1\exists n_2\exists m_2\cdots Q  n_k Q  m_k \left( ( P[n_1,\ldots, n_k])^\perp \;\wedge\; R[m_1, \ldots, m_k]\right),
\end{align}
because $\neg\exists n_1\exists m_1$ is $\forall n_1\forall m_1\neg$,
and because $\slem{k-1}$ is available in $\ha+\slem{k}$. By Lemma~\ref{lem:mqdc}~(\thesecondmqdcassert) and (\thethirdmqdcassert), the $\Pi^0_k$-formula
 \eqref{eq:7} implies a conjunction of two $\Pi^0_k$-formulas.
\begin{align*}
\left(\forall n_1\exists n_2\cdots Q n_k.\ \neg
P[n_1,\ldots, n_k]\right)\quad\wedge\quad\left(\forall m_1\exists
 m_2\cdots Q m_k.\ R[m_1, \ldots, m_k].\right)
\end{align*}  By using assumption \eqref{equiv:4}, the second $\Pi^0_k$-conjunct implies the dual of
the first $\Pi^0_k$-conjunct. So the contradiction follows from Fact~\ref{fac}~({\theassertdualbot}). This
establishes $\dne{k}\vdash_\ha \fpdlem{k}$.

\medskip Next, we prove the converse $\fpdlem{k}\vdash_\ha
\dne{k}$. The axiom scheme $\fpdlem{k}$ has an instance
$\eqref{equiv:4}\to \eqref{eq:5}$ with the $\Sigma^0_0$-formula
$P[n_1,\cdots, n_k]$ being replaced by a false $\Sigma^0_0$-formula
$S(0)=0$. Hence $\ha+\fpdlem{k}$ proves an implication formula $\neg\forall m_1\exists
m _2\cdots \Q m_k.\, R\ \to\ \exists m_1\forall m _2\cdots Q m_k.\, \neg
R$.  So, we can derive $\dne{k}$ by using Modus Tolence if we can prove
an implication formula
\begin{align}
  \alabel{eq:8}
\dneg \exists m_1\forall m_2\cdots Q m_k.\, \neg R\ \to\   \neg\forall m_1\exists m _2\cdots \Q m_k.\, R.
\end{align}
To prove the formula~\eqref{eq:8}, we use a \emph{Gentzen-type sequent
 calculus} $G3$~(see Section~81 of  \cite{MR0051790}) for 
$\IQC$. By the left- and the
right-introduction rules of $\neg$, the $G3$-sequent (\ref{eq:8}) is inferred 
 from a $G3$-sequent
\begin{align*} \exists m_1\forall m_2\cdots Q m_k.\, \neg R,\ \forall m_1\exists m
 _2\cdots \Q m_k.\, R\ \to.\end{align*}
It does not contain the variable $m_1$ free, so it is inferred by the
left-introduction rule of $\exists$ from a sequent 
\begin{align*}
\forall m_2\cdots Q
m_k.\, \neg R,\ \forall m_1\exists m_2\cdots \Q m_k.\, R\
\to.\end{align*}
 It is
inferred by the left-introduction rule of $\forall$ from a $G3$-sequent
\begin{align*}\forall m_2\exists m_3\cdots Q m_k.\, \neg R,\ \exists m_2\forall
m_3\cdots \Q m_k.\, R\ \to.\end{align*}  By repeating this argument, the
$G3$-sequent~\eqref{eq:8} is inferred from a $G3$-sequent $\neg R, R\to$, which is
inferred from an axiom sequent $R\to R$ of $G3$.  This establishes
$\fpdlem{k}\vdash_\ha \dne{k}$, and thus Assertion~\eqref{eq8} of
Theorem~\ref{thm:main}.

\medskip
Assertion~\eqref{eq9} ``$\sdne{k}\vdash_\ha \dlem{k}$'' of Theorem~\ref{thm:main} is proved as follows: By 
 Assertion~\eqref{eq8} of Theorem~\ref{thm:main}, we have $\sdne{k}\vdash_{\ha}
 (A\lequiv B)\to (B\vee \dual{A})$ for any $\Pi^0_k$-formula $A$ and any
 $\Sigma^0_k$-formula $B$.  
By Fact~\ref{fac}~(\theassertdualneg), we have $\sdne{k}\vdash_{\ha}
(A\lequiv B)\to (B\vee \neg A)$.  Thus $\dne{k}\vdash_{\ha} \dlem{k}$. 
This completes the proof of Theorem~\ref{thm:main}.  \bigskip

\begin{remark}\label{rem:negdual}
In $\ha$, the axiom scheme $\dlem{k}$
is strictly weaker than the axiom scheme $\dne{k}$ for every positive integer $k$, according to
\cite{akama04:_arith_hierar_of_laws_of}. Hence there is a
$\Pi^0_k$-formula $A$ such that $\not\vdash_\ha \dual{A} \lequiv \neg
 A$. Otherwise,  by Theorem~\ref{thm:main}~\eqref{eq8}, axiom schemes
 $\dlem{k}$, $\fpdlem{k}$ and $\sdne{k}$ are equivalent over $\ha$.
\end{remark}

The axiom scheme $\sdne{k}$ has the following equivalent axiom schemes.
\begin{fact}\alabel{psdne}For $k\ge0$,
$\dne{k}$ is equivalent in $\IQC$ to  $\pdne{k+1}$.
\end{fact}
\begin{proof} Let an $L_\pra$-formula $\forall n.\, A$
  be a $\Pi^0_{k+1}$-formula with $A$ being a
 $\Sigma^0_k$-formula. We can show
  $\vdash_\IQC \dneg\forall n.\,
 A\to\dneg A$. We have $\sdne{k}\vdash_\IQC \dneg A\to A$.
By Modus Tolence, we have $\dne{k}\vdash_\IQC
 \dneg\forall n.\, A\to A$, and thus  $\dne{k}\vdash_\IQC \dneg\forall n.\, A\to\forall n.\,
  A$. Hence $ \dne{k} \vdash_\IQC \pdne{k+1}$.
To prove the converse $\pdne{k+1}\ \vdash_\IQC \dne{k}$, let $A$ be any
  $\Sigma^0_k$-formula. For any fresh variable $l$, the formula $A$ is
  equivalent in $\IQC$ to a $\Pi^0_{k+1}$-formula $\forall l.\, A$. So
  an instance $\dneg\forall l.\, A\to \forall l.\, A$ of the axiom
  scheme $\pdne{k+1}$ proves in $\IQC$ an instance $\dneg A\to A$ of
  $\sdne{k}$.
\end{proof}

\subsection{Prenex Normal Form Theorem}\label{subsec:pnfthm}

We will introduce three sets of $L_\pra$-formulas such that   the three  correspond to $\Sigma^0_k$-, $\Pi^0_k$-, and
  $\Delta^0_k$-formulas of $\pra$, respectively.

\begin{definition}[$E_k, U_k, P_k$]For the language $L_\pra$, we define $E_k$-, $U_k$-, and
  $P_k$-formulas. 
\begin{enumerate}
\item Given an occurrence of a
  quantifier. If it is in a $\Sigma^0_0$-formula, then we do not assign
      the sign to it. Otherwise,
  \begin{enumerate}
  \item The sign of an occurrence $\exists$ in a formula $A$ is the sign of the subformula
    $\exists n.\, B$ starting with such $\exists$.
  \item The sign of an occurrence $\forall$ in a formula $A$ is the opposite of the sign of the subformula
    $\forall n.\, B$ starting with such $\forall$.
  \end{enumerate}
  
\item The \emph{degree} of a formula is the maximum number of nested
  quantifiers with alternating signs. Formulas of degree 0 are  exactly
      $\Sigma^0_0$-formulas. Clearly the degree is less than or equal to
      the number of occurrences of the quantifiers.

\item By a(n) $U_k$-($E_k$-)formula, we mean a formula of degree $k$ such that all the
  outermost quantifiers are negative~(positive). A $P_{k+1}$-formula is a
  propositional combination of $U_k$- and $E_k$-formulas. 
\end{enumerate}
\end{definition}

The Heyting arithmetic $\ha$ has the function symbols and the defining
 equations for a primitive recursive pairing $p:\Nset^2\to\Nset$ and
 primitive recursive, projection functions $p_0:\Nset\to\Nset$ and
 $p_1:\Nset\to\Nset$ such that $p_0 (p(l,m))=l$, $p_1 (p(l,m))=m$, and
 $p(p_0(n), p_1(n))=n$. It is fairy easy to verify the following fact:

\begin{fact}\label{fact:dupq}An $L_\pra$-formula  $\cdots(\cdots Q l Q
 m\cdots) (\cdots l \cdots m \cdots)\cdots$ is equivalent in $\ha$ to an
 $L_\pra$-formula $\cdots (\cdots Q n\cdots)(  \cdots  (p_0 n) \cdots
 (p_1 n) \cdots )\cdots$ for all $Q\in \{\forall, \exists\}$.
\end{fact}

\begin{theorem}\label{thm:ueps} For any $U^0_k$-($E^0_k$-)formula $A$, we can
 find a $\Pi^0_k$-$(\Sigma^0_k$-, resp.)formula $\hat A$ which is
 equivalent in $\ha + \slem{k}$ to $A$. 
 \end{theorem}
\begin{proof}
The proof is by induction on the structure of $A$.  When $k=0$, we can take $A$ as
$\hat A$ because $A$ is a $\Sigma^0_0$-formula.  Assume $k>0$. Then $A$ is not a prime formula.
 The rest of the proof proceeds by cases according to the form of the formula $A$.

Case~1. $A$ is $B_1 \circ B_2$ with $\circ=\vee,\wedge,\to$

\medskip
 Subcase 1.1 $\circ=\vee,\wedge$. Then $B_1$ and $B_2$ are both
$U^0_k$-($E^0_k$-)formulas. We can use the induction hypotheses  to find two
$\Pi^0_k$-($\Sigma^0_k$-)formulas $\hat B_1$ and $\hat B_2$ which are
equivalent in $\ha + \slem{k}$ to $B_1$ and $B_2$ respectively.

When $A$ is a $U^0_k$-formula, then the $\Pi^0_{k}$-formulas $\hat B_1$ and $\hat B_2$ are
 $\forall l.\, M_1 l$ and $\forall m.\, M_2 m$ for some
$\Sigma^0_{k-1}$-formulas $M_1 l$ and $M_2 m$. Here $M_1 l$ and
$\hat B_2 m$ are both decidable in $\ha + \slem{k}$ because the system $\ha
     + \slem{k}$ proves $\slem{k-1}$ and $\plem{k}$ by 
     Theorem~\ref{thm:main}. So by Lemma~\ref{lem:mqdc} and
 Fact~\ref{lem:int} imply
\begin{align*}\ha + \slem{k}\vdash
A\lequiv \hat B_1\circ\hat B_2\lequiv \forall l(M_1 l\circ \forall
m M_2  m)\lequiv \forall l\forall m(M_1 l\circ M_2  m).
\end{align*} 
Here $M_1 l\circ M_2 m$ is an $E_{k-1}$-formula. By $\slem{k}\vdash_\ha
     \slem{k-1}$, we can use the induction hypothesis to find a
     $\Sigma^0_{k-1}$-formula $\hat D l m$ which is equivalent in $\ha +
     \slem{k}$ to the $E_{k-1}$-formula $M_1 l\circ M_2 m$. So, in $\ha
     + \slem{k}$, the $U^0_k$-formula $A$ is equivalent to $\forall
     l\forall m.\, \hat D l m$ which is equivalent in $\ha + \slem{k}$
     to a $\Pi^0_k$-formula.

When $A$ is an $E^0 _k$-formula, the proof proceeds as in the
 case $A$ is a $U^0_k$-formula.

\medskip Subcase~1.2 $\circ=\to$. Then $B_1$ is an
 $E^0_k$-($U^0_k$-)formula, while $B_2$ is an $U^0_k$-($E^0_k$-)formula.
 We can use the induction hypotheses to find a
 $\Sigma^0_k$-($\Pi^0_k$-)formula $\hat B_1$ and a
 $\Pi^0_k$-($\Sigma^0_k$-)formula $\hat B_2$ such that $\ha +
 \slem{k}\vdash (B_1\lequiv \hat B_1)\wedge (B_2\lequiv \hat
 B_2)$. By Lemma~\ref{lem3.8} and Fact~\ref{fac}~(\theassertdualneg),
 $\ha + \slem{k}\vdash \neg \hat
 B_1\to (\hat B_1)^\bot$. On the other hand, we can show $\IQC \vdash ( \hat
 B_1)^\perp \to \neg \hat B_1$ by using the sequent calculus
 $G3$ for $\IQC$. Hence $\ha + \slem{k}\vdash (\hat B_1)^\perp \lequiv \neg
\hat B_1$.
In $\ha + \slem{k}$, the
 $\Sigma^0_k$-($\Pi^0_k$-)formula $\hat B_1$ is decidable, and thus 
 $(\hat B_1\to \hat B_2)\stackrel{Fact~\ref{lem:int}}{\lequiv} \neg
 \hat B_1\vee \hat B_2 \lequiv (\hat B_1)^\perp\vee \hat B_2$. The two
 disjuncts $(\hat B_1)^\perp$ and $\hat B_2$ are both
 $\Pi^0_k$-($\Sigma^0_k$-)formulas  decidable in
 $\ha+\sdne{k}$. Moreover, each subformula of $(\hat B_1)^\perp$ and
 $\hat B_2$ is so. Hence by Lemma~\ref{lem:mqdc}, Fact~\ref{lem:int} and
 Fact~\ref{fact:dupq}, the formula $(\hat B_1)^\perp\vee \hat B_2$ is
 equivalent in $\ha+\slem{k}$ to a $\Pi^0_k$-($\Sigma^0_k$-)formula.

\bigskip
Case 2. $A$ is $\forall n.\, B[n]\ (\exists n.\, B[n])$. 

Assume $B[n]$ is a $U^0_k$-($E^0_k$-)formula. Then we can find by the induction
hypothesis a $\Pi^0_k$-($\Sigma^0_k$-)formula $\hat B[n]$ which is
equivalent in $\ha+\slem{k}$ to $B[n]$. So, in
$\ha+\slem{k}$, the formula $A$ is equivalent  to $\forall n.\, \hat
 B[n]$ ($\exists n.\, \hat B[n]$), which is equivalent to some $\Pi^0_k$-($\Sigma^0_k$)-formula by
Fact~\ref{fact:dupq}.

Otherwise, $B[n]$ is an $E^0_{k-1}$-($U^0_{k-1}$-)formula. By the
induction hypothesis, we can find a
$\Sigma^0_{k-1}$-($\Pi^0_{k-1}$-)formula $\hat B[n]$ which is equivalent
in $\ha+\slem{k-1}$ to $B[n]$. So,  in $\ha+\slem{k}$, the formula $A$ is equivalent to
$\forall n.\, \hat B[n]$ ($\exists n.\, \hat B[n]$).

\medskip
Case 3. $A$ is $\neg B$. The same argument as Subcase 1.2.
\end{proof}

Here we will prove a slightly stronger version of Theorem~\ref{pnfthm}.
\begin{corollary}\label{cor:1}
For any $P^0_{k+1}$-formula $A$, we can find a $\Pi^0_{k+1}$-formula
 $\hat B$ and a $\Sigma^0_{k+1}$-formula $\hat C$ such that 
$\ha+\slem{k}\vdash A\lequiv \hat B\lequiv \hat C$. Here the number of
 occurrences of quantifiers in $\hat B$ and that of $\hat C$ are less than or equal to that of
 $A$.
\end{corollary}
\begin{proof} By Theorem~\ref{thm:ueps}, the $P^0_{k+1}$-formula $A$ is equivalent in $\ha +
 \slem{k}$ to a propositional combination $A^\circ$ of
 $\Pi^0_k$-formulas and $\Sigma^0_k$-formulas. In the formula $A^\circ$,
 move (0) all the outermost  quantifiers of positive sign, out of all the
 propositional connectives, (1) all the outermost quantifiers of
 \emph{negative} sign, out of all the propositional connectives, (2) all the
 outermost quantifiers of \emph{positive} sign, out of all the propositional
 connectives, (3) all the outermost quantifiers of \emph{negative} sign, out of
 all the propositional connectives, $\ldots$. 
The resulting formula $C$ is a block of quantifiers followed by a
 $\Sigma^0_0$-formula where the block has at most $k+1$ alternations of
 quantifiers~(e.g. If $A$ is a $P^0_2$-formula $\forall x P x \wedge
 (\exists y P' y \to \exists z P''z)$ with $P, P', P''$ being $\Sigma^0_0$-formulas, then
 $A^\circ$ is $\exists  z \forall x y (P x \wedge (P' y \to P'' z$)) which has 2
 alternations of quantifiers). 
All the $\Pi^0_k$- and all
 the $\Sigma^0_k$-formulas are $(\ha+\slem{k})$-decidable. So, by
 Lemma~\ref{lem:mqdc} and Fact~\ref{lem:int}, the formula $C$
 is equivalent in $\ha+\slem{k}$ to the $P^0_{k+1}$-formula $A$. By
 Fact~\ref{fact:dupq}, the resulting formula is equivalent in
 $\ha+\slem{k}$ to a $\Sigma^0_{k+1}$-formula $\hat C$. In a similar
 way, the $P^0_{k+1}$-formula $A$ is equivalent in $\ha+\slem{k}$ to a
 $\Pi^0_{k+1}$-formula $\hat B$.
\end{proof}

\section{Iterated Autonomous Limiting PCAs}\alabel{sec:limitingPCA}

We recall \emph{autonomous limiting} \textsc{pca}s~\citep{MR2030297}. The
construction was based on the Fr\'echet filter
on $\Nset$, and is similar to but easier than the constructions of \emph{recursive
ultrapower}~\citep{MR0381969} and then  semi-ring made from recursive
functions modulo co-$r$-maximal sets~\citep{MR0265157}.  

 We say a partial numeric function $\varphi(n_1,\ldots,n_k)$
is \emph{guessed} by a partial numeric function   
$\xi(t,n_1,\ldots,n_k)$ as $t$ goes to infinity, provided that 
$\forall
  n_1,\ldots,n_k
  \exists t_0\forall t>t_0.\ \varphi(n_1,\ldots,n_k)\simeq
  \xi (t, n_1,\ldots,n_k)$. Here, the relation $\simeq$ means ``if one side is
  defined, then the other side is defined with the same value.'' In this case, we write 
$\varphi (n_1,\ldots,n_k)\simeq
  \lim_t \xi(t, n_1,\ldots,n_k)$. On the other hand, the symbol `='
  means both sides are defined with the same value.
  For every class $\mathcal{F}$ of partial numeric functions, 
  $\lim(\mathcal{F})$ denotes the set of partial numeric functions guessed
  by a partial numeric function in $\mathcal{F}$.

A partial combinatory algebra~(\textsc{pca} for short) is a partial algebra
$\A$ equipped with two distinct
constants $\k,\s$ and a partial binary operation
``application'' $(-)\cdot(\bullet)$ subject to 
$ (\k\cdot a)\cdot b = a$, $((\s \cdot a)\cdot b )\cdot z \simeq (a
\cdot c)\cdot (b\cdot c)$, and $(\s\cdot a)\cdot b$ is defined. 
We introduce the standard convention of associating the application to
the left and writing $a b$ instead of $a\cdot b$, omitting parentheses
whenever no confusion occurs. If $a\cdot
b$ is defined then both of $a$ and $b$ are defined.

\emph{The $0$-th Church numeral} of $\A$ is an element $\k\; (\s
\; \k \; \k)$
of $\A$.  \emph{The $(n+1)$-th Church numeral} of $\A$ is an element $\s\;
( \s \; (\k\; \s )\; \k )\; \num{n}$ of $\A$.  By definition, for each
natural number $n$, an element $\num{n}$ of $\A$
\emph{represents} $n$,  and an element $a$ of $\A$
\emph{represents} itself.
We say a
partial function
$\varphi$ from $M_1\times M_2\times \cdots\times M_k$ to $M_0$
is represented by an element $a$ of $\A$, whenever 
 $\varphi(x_1,\ldots, x_k)=x_0$ if and only if for all representatives $a_i\in\A$ of
 $x_i$ ($1\leq i\leq k$),
$a\;a_1\;\cdots\;a_{k-1}\; a_k$ is defined and is a representative of $x_0$. 
The set of $\A$-representable partial
functions from  $M$ to $M'$ is denoted by $M{\rightharpoonup_{\A}}\; M'$.
Each partial recursive function is representable in any \textsc{pca}.

Let $\sim $ be the partial equivalence relation on $\A$ such that $a\sim
b$ if and only if
$a\; \num{t}= b\; \num{t}$ for all but finitely many
natural numbers $t$.  A quotient structure
$(\Nset{\rightharpoonup_{\A}}\;\A)/\sim$ will be a \textsc{pca} by the
argument-wise application operation modulo $\sim$. More precisely, let
$\yohji{a}$ be $\{b\in \A\;|\; b\sim a\}$. Then the set
$\{\yohji{a}\;|\; a\in\A\ \mbox{and}\ a \sim a\}$, 
$\bm{k}:=\yohji{\k\; \k}$, $\bm{s}:=\yohji{\k\; \s}$ and the following
operation $ \yohji{a}\ast\yohji{b} \simeq \yohji{\s\; a\; b}$
defines a \textsc{pca}. We denote it by $\plim(\A)$.

By a \emph{homomorphism} from a \textsc{pca} $\A$ to a \textsc{pca}
$\B$, we mean a function from $\A$ to $\B$ such that $f(\k)=\k$,
$f(\s)=\s$, and
$f(a)\; f( b)\simeq f(a\; b)$ for all $a,b\in \A$.
A homomorphism   fits in with a
``\emph{strict, total
homomorphism} between \textsc{pca}s''~(see p.~23 of \cite{Hofstra:2010:UTU:1857260.1857303}).
A \emph{canonical injection} of a \textsc{pca} $\A$ is, by definition, an injective homomorphism
$\iota_\A :\A\to\plim(\A)\;;\; x\mapsto
\yohji{\k\; x}$.

\begin{fact}\label{fact:caninj}
 $\iota_{\A}$ is indeed an injective homomorphism for every \textsc{pca}
 $\A$.
\end{fact}
\begin{proof} We can see that $\iota_\A$ 
 is indeed a function from $\A$ to $\plim(\A)$. 
In other words,
 $\iota_\A$ is ``total'' in a sense of
 \cite{Hofstra:2010:UTU:1857260.1857303}. It is proved as follows:
 For every $x\in \A$, we have $\k\;  x \; \overline{t} =
\k \; x \; \overline{t}$ for every $t\in\Nset$. This implies
$\k\; x\sim \k\; x$, from which $\iota_\A(x)=\yohji{\k\;x}$ is in
$\plim(\A)$. The function $\iota_\A$ is injective, because
$\iota_\A(x)=\iota_\A(y)$ implies $\k\; x\; \num{t} = \k\; y\;\num{t}$
for all but finitely many natural numbers $t$, from which $x= \k\; x\;
\num{t} = \k\; y\;
\num{t}  = y$ holds for some natural number $t$. 

It holds that (i) the injection
$\iota_A$ maps the intrinsic constants $\k,\s$ of the \textsc{pca} $\A$
to $\bm{k},\bm{s}$ of the \textsc{pca} $\plim(\A)$, and (ii) $\iota_\A(a)\;
\iota_\A(b)\simeq \iota_\A(a\; b)$. In other words, the injection
$\iota_\A$ is ``strict'' in a sense of
\cite{Hofstra:2010:UTU:1857260.1857303}. The Assertion~(i) is clear by
the definition. As for the Assertion~(ii),
we can prove that if $\iota_\A(a\; b)$ is defined then $\iota_\A(a)\; \iota_\A(b)$ is defined with the
same value. The proof is as follows: By the premise, $a\; b$
is defined. Because $\k\; (a\; b)\; \overline{t}= (a\; b)
=\s\; (\k \; a)\; (\k\;  b)\;
\overline{t} $ for all $t\in\Nset$, we have 
\begin{align}\iota_\A(a)\;
\iota_\A(b)\simeq \yohji{\s\;(\k\; a)\; (\k\; b)} \simeq \yohji{\k \; (a \; b)}\simeq
\iota_\A(a\; b). \label{uum}
\end{align}
We can prove that if $\iota_\A(a)\; \iota_\A(b)$ is defined then  $\iota_\A(a b)$  is defined with the
same value. The proof is as follows: By the premise, $\yohji{\s\; (\k\; 
a)\;  (\k \;  b)}$ is defined. So $\s\; (\k\; 
a)\;  (\k \;  b)\sim \s\; (\k\; 
a)\;  (\k \;  b)$. Hence
for all but finitely many natural numbers
$t$, $\s\; (\k\; 
a)\;  (\k \;  b) \;  \overline{t}\simeq a\; b$ is defined. Thus $(a \; b)$
is defined. By \eqref{uum}, the Assertion~(ii) follows.
\end{proof}

Because $\iota_\A$ is a homomorphism, we have $\overline{n}^{\plim(\A)} = \iota_\A(\num{n})$.
Hence the limit is the congruence class of the guessing
function, as follows:
\begin{align}\lim_t\left(\xi\; \overline{t}\right) = 
\overline{n}\ \mbox{in}\ \A \iff \yohji\xi = \overline{n}\ \mbox{in}\
\plim(\A). \qquad(\xi\in\A) \label{equiv:lim}
\end{align}
 The direct limit of
$\A\stackrel{\iota_\A}{\to}\plim(\A)\stackrel{\iota_{\plim(\A)}}{\to}\plim^2(\A)\cdots$
is indeed a \textsc{pca}, and will be denoted by $\plim^\omega(\A)$.
The application operator of a \textsc{pca} and
  ``limit procedure'' commute; 
\begin{align*}(\lim_t a\;  \overline{t}) \ast (\lim_t b\; \overline{t}) =
  \yohji{a}\ast \yohji{b}= \yohji {\s\;  a\;  b} = \lim_t \s\; 
  a\;  b\;  \overline{t} = \lim_t (a \;   \overline{t}) \; 
 (b\;  \overline{t}).
\end{align*} 

The set of partial numeric functions represented by a \textsc{pca} $\A$
is denoted by $\Repr(\A)$.  By the bounded maximization of a function
$f(x,\vec{n})$, we mean a function $\max_{x<l }f(x, \vec{n})$. The
following fact is
well-known.
\begin{fact}\label{fact:repr}For every \textsc{pca} $\B$, the set of
 functions represented by elements of $\B$ is closed under the
 composition, the bounded maximization and under $\mu$-recursion.  \end{fact}
Then, we can prove
$\Repr(\plim^\alpha(\A))=\cup_{n<\max(1+\alpha,\omega)}\lim^n(\Repr(\A))$.
 Shoenfield's limit lemma~(see \cite{MR982269} for instance) implies that the
\textsc{pca} $\plim^\alpha(\A)$ represents all
$\emptyset^{(\max(\alpha,\omega))}$-recursive functions. So, the
\textsc{pca} $\plim^\omega(\A)$ can represent any arithmetical function.

\section{Iterated Limiting Realizability Interpretation of Semi-classical EONs}\alabel{sec:classical}

It is well-known that a form of Markov Principle over the language $L_\pra$,
\begin{displaymath}
\dne{1}\quad\dneg\exists n \forall m<t. f(n,m,l)=0\to\exists
n\forall m<t. f(n,m,l)=0
\end{displaymath}
is realized by an ordinary program $r(t,l)=\mu n. \max_{m<t}f(n,m,l)=0$ via recursive
realizability interpretation of \cite{MR0015346}. Here the program $r(t,l)$
is representable by a \textsc{pca} $\A$.
A stronger principle of classical logic
\begin{displaymath}
\dne{2}\qquad \dneg\exists n \forall m.\, f(n,m,l)=0\to\exists n \forall
  m.\, f(n,m,l)=0 ,
\end{displaymath} 
the ``limit'' with respect to $t$ of a $\dne1$, turns out to be realized by
a limiting computation $\lim_t r(t,l)$ which  is representable
by a limiting \textsc{pca} $\plim\A$.
This simple approach can be extended to
an \emph{iterated limiting realizability interpretation} of $\dne\alpha$ for
$\alpha\leq\omega$, by $\plim^\alpha \A$.

For the convenience, we embed $\ha+\dne{1+\alpha}$ in a corresponding
extension of a constructive logic $\eon$. It is $\eon$ plus a form of $\dne{1+\alpha}$. The
iterated limiting realizability interpretation is introduced by using an
$\alpha$-iterated autonomous limiting \textsc{pca}s
$\plim^\alpha(\A)$.

Here $\eon$ is a constructive logic of partial
terms~(see p.~98 of \cite{MR786465}), and the language includes Curry's combinatory constants, and a partial
application operator symbol. 
The language of $\eon$ is
$\{(-)\cdot(\bullet),\ \s, \k, \cased, 0, \suc, \pred, \p, \car,\cdr; =, N,
\downarrow\}$. 
Here the constant symbols $\p,\car,\cdr$ are intended to be the pairing function, the first projection, and the second
projection, respectively. The predicate symbol $=$ means ``the both hand sides are
defined and equal.'' The 1-place predicate symbols $N$ and $\downarrow$ mean
``is a natural number'' and ``is defined,'' respectively. As before, we write
$ a_0 \;  a_1 \;  a_2\;  \cdots a_{n-1}\;  a_n$ for
$(\cdots ((a_0 \cdot a_1 )\cdot a_2)\cdot \cdots a_{n-1})\cdot a_n$,
whenever no confusion occurs.

In writing formulas of $\eon$, variables $n,m,l,i$ and $j$ will be
implicitly restricted to the predicate $N$, i.e. they are ``natural
number variables.'' So, $\forall n.\, A n$ is the abbreviation for
$\forall x.\, (N x\to A x)$ and $\exists m.\, B m$ for $\exists y.\, (N
y\wedge B y)$. We review the logical axioms of $\eon$ from p.~98 of \cite{MR786465}.
The logical axioms and rules of $\eon$ are as follows: $\eon$ has
the usual propositional axioms and rules. The quantifier axioms and
rules are as follows: From $B\to A$ infer $B\to \forall x A$ ($x$ not
free in $B$).  From $A\to B$ infer $\exists x A\to B$ ($x$ not free in
$B$). $\forall x A[x] \wedge t\defined \to A[t]$.  $ A[t] \wedge
t\defined \to\exists x A[x]$. $x=x$. $x=y  \to y=x$. $t=s \to t\defined
\wedge s\defined$. $R(t_1,\ldots, t_n)\to t_1\defined\wedge\cdots \wedge
t_n\defined$. ($R$ is any atomic formula). $c\defined$ (every constant
symbol $c$). $x\defined$ (every variable $x$). Let us abbreviate
$t\simeq s$ for $(t\!\defined \vee\; s\!\defined \ \to\ t=s)$. $\eon$ has a logical
axiom $t\simeq s\to A[t]\to A[s]$.

The non-logical axioms of $\eon$ consists of
\begin{align*}
& \k x y = x,\quad \s x y z \simeq x z (y z),\quad \s x y \downarrow, \quad \k \ne
 \s,\\
& \p x y \downarrow,\quad \car (\p x y ) = x ,\quad \cdr (\p x y) = y,\\
&N(0),\quad \forall x \left(N x \to [N(\suc x) \ \wedge\  \pred (\suc x) = x\
 \wedge\ \suc x \ne 0]\right),\\
& \forall x \left(N x \ \wedge\ x\ne 0 \to N(\pred x) \ \wedge\  \suc(\pred x) = x\right),\\
& N x\ \wedge\ N y\ \wedge\ x = y \to \d x y u v=u,\\
& N x\ \wedge\ N y\ \wedge\ x \ne y \to \d x y u v=v,\\
& A(0)\ \wedge\ \forall x \left( N x\ \wedge\ A(x)\to A(\suc x)\right) \to
\forall x(N x\to A(x)).
\end{align*}

\emph{We will interpret $\eon$ in a \textsc{pca}, as we interpret classical
logic in a model theory}.  The interpretations of the constant symbols
$\s,\k$ are the corresponding constants of the \textsc{pca} $\A$. The
interpretations of the constant symbols $0$, $\pred$, $\suc$, $\cased$
in $\A$ are defined in a similar way that they are represented in
Curry's combinatory logic by Church numerals. The interpretation of the
pairing $\p$ and projections $\car,\cdr$ are as in Curry's combinatory
logic. For detail, see \cite{BlueBook}. The application operator symbol
$(-)\cdot(\bullet)$ of $\eon$ is interpreted as the application of the
\textsc{pca} $\A$. The unary predicate symbols $N$ and $\defined$ are
interpreted as the set of Church numerals of $\A$ and $\A$ itself,
respectively. The binary predicate symbol $=$ is interpreted as just the
identity relation on $\A$. Given an assignment
$\rho:\{\eon\mbox{-variables}\}\to \A$. The interpretation of an
$\eon$-term $t$ in $\A$ and $\rho$ is defined as an element of $\A$ as
usual. The interpretation of an
$\eon$-formula $A$ in the \textsc{pca} $\A$ and $\rho$ is defined as usual as one of the
truth-value $\top,\bot$.
We say an $\eon$-formula $A$ is true in a \textsc{pca} $\A$ and an
assignment $\rho:\{\eon\mbox{-variables}\}\to\A$, if the
interpretation of $A$  in $\A$ and $\rho$ is $\top$. In this
case we write $\A,\rho\models A$. If $\A,\rho\models A$ for every
$\rho$, then we write $\A\models A$.

\begin{definition}\alabel{eonrealizability}Let $T$ be a formal system extending $\eon$. The realizability interpretation of $T$ is just an association to each
formula $A$ of $T$ another formula $\exists e.\,\realize{e}{A}$ of $T$ with a
variable $e$ being fresh. It is read ``some $e$
realizes $A$.''  
\def\e{t}
For an $\eon$-term $t$ and an $\eon$-formula $A$, we define an $\eon$-formula
 $\realize{t}{A}$ as follows:
\begin{itemize}
\item $\realize{\e}P$ is $\e\downarrow\wedge\; P$ for each  atomic formula $P$.
\item $\realize{\e}{\neg A}$ is $\e\downarrow\wedge\;\forall x(\neg \realize{x}{A})$.
\item $\realize{\e}{A\to B}$ is $\e\downarrow\wedge\;\forall      x(\realize{x}{A}\to \e x\downarrow\;
  \wedge\; \realize{\e x}{B})$.
\item $\realize{\e}{\forall x.\, A}$ is $\forall x(\e x\downarrow\;
  \wedge\;\realize{\e x}{A})$.
\item $\realize{\e}{\exists x.\, A[x]}$ is
 $\realize{\cdr \e}{A[\car \e]}$. 

\item $\realize{\e}{A\wedge B}$ is $\realize{\car \e}{A}\wedge \realize{\cdr \e}{B}$.

\item $\realize{\e}{A\vee B}$ is $N(\car \e)\wedge (\car \e = 0\to
  \realize{\cdr \e}{A})\wedge(\neg \car \e = 0\to \realize{\cdr \e}{B})$.
\end{itemize}
\end{definition}

\begin{definition}\label{def:naive}
 A formal arithmetic
 $T$ extending $\eon$ is said to be \emph{sound} by  the realizability interpretation
 for a \textsc{pca} $\A$, provided that for every sentence $B$ 
provable in $T$, a sentence  $\exists e.\,(\realize{e}{B})$ is true in 
 $\A$.
\end{definition}

(Realizability) interpretations and model theory of a (constructive)
arithmetic $T$ are often formalized within the system $T$ plus reasonable axioms. For
example, \cite{MR48:3699}, \cite{avigad00} and so on formalized
realizability interpretations of constructive logics, while
\cite{smorynski78:_nonst_model_of_arith}, \cite{MR1748522} and so on did
non-standard models of various arithmetic. However, as we defined in
Definition~\ref{def:naive}, we will carry out our
realizability interpretation within a naive set theory. This readily leads to
the second assertion of the following Lemma.

\begin{lemma}\label{lem:realpnftrue} 
Suppose $B$ is an $\eon$-formula in \textsc{pnf} with all the variables
 relativized by the predicate $N$. 
\begin{enumerate}
\item \label{assert:realpnftrue1}For any  $\eon$-term $t$,  we have
      $\eon\vdash \realize{t}{B} \to B$.

\item \label{assert:realpnftrue2} If $B$ is an $\eon$-sentence and  $\A$
      is a \textsc{pca}, then $\A\models\dneg \exists
      x.\ \realize{x}{B}$ implies $\A\models B$.
\end{enumerate}
\end{lemma}
\begin{proof}\eqref{assert:realpnftrue1} The proof is by induction on the structure of $B$.
 When $B$ is prime, it is trivial.
 When $B$ is $\forall x(N x \to A x)$, then 
 $\realize{t}{B}$ is $\forall x (t\cdot x\defined \ \wedge\ \forall y(N
 x \to t\cdot x\cdot y\defined\ \wedge\ \realize{t\cdot x\cdot
 y}{A x}))$ where the variables $x$ and $y$ are fresh. So the induction
 hypothesis implies $\realize{t}{B} \to \forall x (t\cdot x\defined \ \wedge\ \forall y(N
 x \to t\cdot x\cdot y\defined\ \wedge\ A x))$. Because $y$ is fresh,
 $\realize{t}{B} \to \forall x(N x \to A x)$.
 When $B$ is $\exists x(N x \wedge A x)$, then 
 $\realize{t}{B}$ is $\car  (\cdr t)\defined\ \wedge\ N(\car  t)\ \wedge\
\realize{\cdr(\cdr t)}{A (\car t)}$. So the induction
 hypothesis implies $\realize{t}{B} \to \car  (\cdr t)\defined\ \wedge\ N(\car  t)\ \wedge\
{A (\car t)}$. Hence,
 $\realize{t}{B} \to N(\car  t) \ \wedge\  A (\car  t)$. Thus $\realize{t}{B}\to \exists x(N x\
 \wedge\ A x )$.

\eqref{assert:realpnftrue2} By Definition~\ref{eonrealizability}, the
 system $\eon$ proves a sentence $\exists x.\ \realize{x}{\dneg
 B}\to \dneg \exists x. \realize{x}{B}$. By the premise and the
 soundness of $\eon$ for any \textsc{pca}, $\dneg \exists
 x. \realize{x}{B}$ is true in the \textsc{pca} $\A$, and thus $\exists
 x. \realize{x}{B}$ is so.  By the soundness of $\eon$ in any
 \textsc{pca} and the
 Assertion~\eqref{assert:realpnftrue1} of this Lemma, the sentence $B$
 is true in the \textsc{pca}.\end{proof}

We will make the argument of the first paragraph of this section
rigorous.  It is instructive to consider the following Lemma.

\begin{lemma}\label{lem:jump}
 For each closed $\eon$-term $t$ and for each \textsc{pca} $\A$, whenever  $\A\models \forall m_1\forall m_2.\,
  N(t\;m_1\;m_2)$ holds, it holds
\[
\plim (\A) \models    \exists x.\,\bigl[\realize{x}{(\dneg\exists m_1\forall m_2.\,t\;m_1\;m_2 =0\to \exists m_1\forall m_2.\,t\;m_1\;m_2 = 0)}\bigr].
\]
\end{lemma}
\begin{proof} 
 Let an $\eon$-formula
$\realize{q}{\dneg\exists m_1\forall m_2.\,t\;m_1\;m_2 =
  0}$
be true in $\plim(\A)$. By Lemma~\ref{lem:realpnftrue}~\eqref{assert:realpnftrue2},
for some natural number
$n_1$, the $\eon$-sentence $\forall  m_2.\, t \;\overline{n_1}\; m_2 = 0$  is true in $\plim(\A)$.

We can see that  $\A$ has an element $\xi$
representing the following unary numeric function:
\begin{displaymath}
\g(l):={\mu m_1.\, \bigl((\max_{m_2 < l}\ t\; m_1\; m_2) = 0\bigr)}.
\end{displaymath}

Note that $\g ( l)\leq \g (l')\leq n_1$ if $l\leq l'$. So, some natural
number $m_1$ satisfies $\lim_l \g ( l) = m_1$. That is, for all natural
numbers $l$ but finitely many, we have $\g ( l ) =m_1$. So, for all
natural numbers $l$ but finitely many, the formula $\xi\; \overline{l} =
\overline{m_1}$ is true in $\A$.

 By the definition of $\plim(\A)$, we have 
 \begin{align*}\yohji\xi =
 \overline{m_1}
 \end{align*} in $\plim(\A)$.  By the definition of $\xi$, for all natural
 numbers $l$ but finitely many, an $\eon$-sentence  $(\max_{m_2 < l }\; t\;
 \overline{m_1}\; \overline{m_2}) = 0$ is true in $\plim(\A)$.
Therefore, for all natural numbers
$m_2$, an $\eon$-sentence  $t\; \overline{m_1}\; \overline{m_2}=
 0$ is true in $\plim(\A)$. Hence, $ \forall m_2.\, t\; \overline{m_1}\; m_2 =0$ is true in $\plim(\A)$.

So, as a realizer $x$ of $\neg\neg \exists m_{1} \forall m_{2}.\ t\;m_{1}\; m_2 =0\ \to \ \exists m_{1} \forall m_{2}.\ t\;m_{1}\; m_2 =0$, take 
$\k \Bigl(\p\ \bigl(\p\; 0\; (\k (\k\; 0))\bigr)\ \yohji\xi\Bigr) \in \plim(\A)$.
\end{proof}

\begin{definition}\label{def:sdneprime} For each \textsc{pca} $\A$ and each nonnegative integer $k$,
$(\sdneprime{k})$ is a rule
\begin{align*}
\frac{\mbox{$t$ is a closed term of $\eon$} \qquad\qquad \forall \vec{n}\forall m_1\ldots\forall m_k.\,  N( t\ \vec{n}\ m_1\ \cdots\ m_k)}
{\forall \vec{n}\left(\begin{array}{ll}\dneg\exists m_1\forall m_2\exists m_3\cdots Q_{k}
 m_k.\ t\;\vec{n}\; m_1\; m_2 \ldots  m_k=0\\
\,\to\exists m_1\forall m_2\exists m_3\cdots Q_{k} m_k.\ t\;\vec{n}\;
		      m_1\; m_2 \ldots  m_k=0\end{array}\right)}
\end{align*}
Here $Q_k$ is $\exists$ for odd $k$ and $\forall$ for even $k$.
\end{definition}

%We consider a game an $\eon$-formula $\exists m_1 \forall m_2.\ t\;m_1\;m_2=0$ represents. In the
%game, the proponent $\exists$ plays once and the opponent $\forall$
%does so.
%%
%The proponent $\exists$ can compute a minimal move $\hat m_1$, if the
%opponent $\forall$ chooses the next move $m_2$ from finitely many
%candidates, say $[0, l]\cap \Nset$.  Then the proponent $\exists$
%computes the minimal move $\hat m_1=\g_1(l)$  from the
%bound $l$ of opponent's moves. However, if there is a winning strategy for the proponent
%$\exists$, then there is a move $n_1$ of the proponent $\exists$ that leads to the
%proponent's win. Actually,
%the proponent $\exists$ can produce the \emph{minimum} winning move $m_1$, as the limit of $\hat m_1(l)$ in the limit of $l\to\infty$.
%
%For the game an $\eon$-formula in \textsc{pnf} represents, 
%by constructing a winning strategy of the proponent $\exists$ in the limit, 

\begin{theorem}\label{thm:eonrealize} For each nonnegative integer $k$ and each \textsc{pca} $\A$, if the system $\eon$ +
$(\sdneprime{k+1})$ proves an $\eon$-sentence $A$, then a sentence $\exists e.\ \realize{e}{A}$ is true in the \textsc{pca}
  $\plim^k(\A)$. 
\end{theorem}

\begin{proof} 
The verification is by induction on the length of the proof $\pi$ of
$A$. The axioms and rules other than $(\sdneprime{k})$ is
manipulated as in the proof of Theorem 1.6 of \cite{MR786465}. 

We will consider the case $(\sdneprime{k})$. By the induction hypothesis on the proof $\pi$, an $\eon$-sentence $\exists e.\ \realize{e}{\forall \vec{n}\forall m_{1} \ldots \forall m_{k}. \, N(t\ \vec{n}\ m_{1} \ \cdots\ m_{k})}$ is true in the \textsc{pca} $\plim^k(\A)$.  We will derive that an $\eon$-sentence 
\begin{align*}\exists e.\ \realize{e}{\forall \vec{n}(\dneg\exists m_1\forall m_2\cdots Q_{k}
 m_k.\ t\;\overline{\vec{n}}\; m_1\; m_2 \ldots  m_k=0\\
  \to\ \exists m_1\forall m_2\cdots Q_k m_k.\ t\;\overline{\vec{n}}\;
		      m_1\; m_2 \ldots  m_k=0)}\end{align*}
is true in $\plim^{k}(\A)$. 
 Let  $x$ be an element of $\plim^{k}(\A)$ and $\vec{n}$ be nonnegative integers. Suppose 
 \begin{align*}\plim^{k}(\A)\models\realize{x}{\dneg\exists m_1\forall m_2\cdots Q_{k}
 m_k.\ t\;\overline{\vec{n}}\; m_1\; m_2 \ldots  m_k=0}.\end{align*} 
By Lemma~\ref{lem:realpnftrue}~\eqref{assert:realpnftrue2}, we have
\begin{align*}\plim^{k}(\A)\models Q_1 m_1  Q_2 m_2  Q_3 m_3
 \cdots Q_k m_k.\  t\;\overline{\vec{n}}\;m_1\cdots m_k=0.
 \end{align*} For every closed $\eon$-term $t'$, the valuation of $t'$  in $\plim^{k}(\A)$ is obtained from the valuation of $t'$ in $\A$ by the canonical injection $\iota_{\plim^{k-1}(\A)}\circ\cdots\circ\iota_{\A}$. Hence
\begin{align}
\A\models{ Q_1 m_1  Q_2 m_2  Q_3 m_3
 \cdots Q_k m_k.\  t\;\overline{\vec{n}}\;m_1\cdots m_k=0}\label{hosi}
\end{align} 
where $Q_i=\exists\ (i:\mbox{odd});\ \ \forall\ (i:\mbox{even})$.

\begin{definition}\label{equiv:gj} For each \textsc{pca} $\A$ and each $j=0,\ldots,k-2$,  define a total function $g_j(\vec{n},
\nu_1,\ldots,\nu_{k - j})\ :\ \vec{\Nset}\times\Nset^{k - j}\to \{0,
1\}$ such that
\begin{align*}
 g_j(\vec{n}, \nu_1,\ldots,\nu_{k - j})=0 
\iff\A\models\left\{\begin{array}{r}
(Q_{k - j + 1} m_{k - j +1})\  (Q_{k - j + 2}
 m_{k - j + 2})\  \cdots (Q_k m_k).\\
	   t \ \overline{\vec{n}}\ \overline{\nu_1}\ \cdots \ \overline{\nu_{k - j}} \ m_{k - j + 1}\cdots
 m_k = 0.
\end{array}\right. 
\end{align*}\end{definition}

\begin{claim}\label{claim:represent}For each $j=0,\ldots,k-2$,  the
total function $g_j$ is   represented by some element of a \textsc{pca} $\plim^{j} \A$.\end{claim}
\begin{proof}We can define  $g_j$ as a $j$-nested  limiting function, as follows:
\begin{align*}
 g_0(\vec{n}, \nu_1,\ldots,\nu_k) &:= \min(1,l) \ \mbox{such that}\ \A\models t\
 \overline{\vec{n}}\ \overline{\nu_1}\ \cdots\  \overline{\nu_k}  =
 \overline{l} .\\
g_j (\vec{n}, \nu_1,\ldots,\nu_{k - j}) &:= \begin{cases}
					      \displaystyle \lim_{l }
						\max_{\nu_{k- j +
						1}<l } g_{j -
						1}(\vec{n}, \nu_1,\ldots,
						\nu_{k - j +1}), &
						(\mbox{$k - j$ is odd});\\
					     \displaystyle     \lim_{l }
					     \min_{\nu_{k - j + 1}<l }
					     g_{j - 1}(\vec{n},
					     \nu_1,\ldots, \nu_{k - j +1}),
					     & (\mbox{$k - j$ is
					     even}).\end{cases}
\end{align*} 
The claim is derived from \eqref{hosi} by induction on
 $j$, because $g_j$  is the limit of a bounded  monotone function which
 is either $\max_{...<l}$ or
 $\min_{..<l}$. Each $g_j$ is represented by some element of a
 \textsc{pca} $\plim^j \A$, because of \eqref{equiv:lim}. This completes the proof of Claim~\ref{claim:represent}.\end{proof}

We continue the proof of Theorem~\ref{thm:eonrealize}. For an $\eon$-formula 
\begin{align}\exists m_1 \forall m_2 \exists
m_3\cdots Q_k m_k.\ t \;\overline{\vec{n}}\; m_1\;\cdots m_k=0 \label{eonformula:hosi}
\end{align} 
appearing in \eqref{hosi},  consider the ``game'' represented by \eqref{eonformula:hosi} between the proponent $\exists$ and the opponent $\forall$. 
From any moves
 $\nu_2,\nu_4, \ldots,\nu_{2p - 2}$  $(p=1,2,\ldots, \lfloor (k+2)/2 \rfloor)$  taken by the opponent $\forall$, 
 the \emph{minimum} move
 $m_{2p-1}(\vec{n}, \nu_2,\nu_4,\ldots,
 \nu_{2p-2})$ by the proponent
 $\exists$ is given by
the following limiting function 
\begin{definition}\label{def:propmove}For $p=1,2,\ldots,\lfloor (k+2)/2
 \rfloor$, let
\begin{align*}
m_{2p-1}(\vec{n}, \nu_2,\nu_4,\ldots,
 \nu_{2p-2}):=\lim_l \ \g_{2p-1}(l , \vec{n}, \nu_2,\nu_4,\ldots,
 \nu_{2p-2}).
\end{align*}
Here the guessing function  $\g_1(l,\vec{n})=\mu m_1(\max_{\nu_2<l}
 g_{k-2}(\vec{n}, m_1, \nu_2))$ is obtained from $g_{k-2}$ by the
 bounded maximization $\max_{\nu_2<l}$ and the $\mu$-recursion. For
 $p>1$, define the function $\g_{2p-1}$ by the composition, the bounded
 maximization $\max_{\nu_{2p}<l}$ and the $\mu$-recursion $\mu m_{2p-1}$.
\begin{align*}
\g_{2p-1}(l , \vec{n}, \nu_2,\nu_4,\ldots,
 \nu_{2p-2})\\
:=
\mu m_{2p-1}.\ \Bigl(\max_{\nu_{2p}<l} g_{k-2p}\bigl(\vec{n},\ &
 m_1(\vec{n}),\ \nu_2,\ m_3(\vec{n}, \nu_2),\ \nu_4,\
 m_5(\vec{n}, \nu_2, \nu_4),\ldots,\\
\ \  &m_{2p-3}(\vec{n},
 \nu_2,\ldots, \nu_{2p-4}), \ \nu_{2p-2},\ m_{2p-1},\ \nu_{2p}\bigr) = 0\Bigr).
\end{align*}
\end{definition} 

For the function $m_{2p-1}$  defined above,  we have the following:
\begin{claim} \label{claim:moverepr}
Assume $p=1,2,3,\ldots,\lfloor (k+2)/2 \rfloor$. Then the
 following assertions hold:
\begin{enumerate}\item
\label{claim:minimum}
$m_{2p-1}(\vec{n}, \nu_2,\nu_4,\ldots, \nu_{2p-2})$ is indeed a total
 function of $\vec{n}, \nu_2,\nu_4,\ldots, \nu_{2p-2}$. 
 
For the game the $\eon$-formula \eqref{eonformula:hosi} represents, consider the following alternating sequence $\sigma$ of the proponent $\exists$'s moves and the opponent
 $\forall$'s moves of the game:
\begin{align*}
( m_1(\vec{n}),\ \nu_2,\ m_3(\vec{n}, \nu_2),\ \nu_4,\
 m_5(\vec{n}, \nu_2, \nu_4),\ldots, m_{2p-3}(\vec{n},
 \nu_2,\ldots, \nu_{2p-4}),  \nu_{2p-2}  )\in \Nset^{2p-2}
 \end{align*} 
  Suppose that $n_{2p-1}\in \Nset$ is a proponent's move that immediately follows 
 the sequence $\sigma$. Then  $n_{2p-1}\ge m_{2p-1}(\vec{n},
		      \nu_2,\nu_4,\ldots, \nu_{2p-2})$.

\item \label{assert:mrepresent}
 The limiting function $m_{2p-1}$ is represented by
 an element of a \textsc{pca} $\plim^k (\A)$.

\end{enumerate}
\end{claim}

\begin{proof} \eqref{claim:minimum} 
The proof is by induction on $p$. 
 The case where $p=1$
 is essentially due to the
 proof of Lemma~\ref{lem:jump}. Let $p>1$. Assume  (i) the opponent's
 $2p$-th move $\nu_{2p}$ is bounded from above by $l$, (ii) the parameter
 $\vec{n}$ of the game is supplied, and (iii) the opponent's moves
 $\nu_2,\nu_4,\ldots, \nu_{2p-2}$ so far are supplied. By the induction hypotheses, the functions
 $m_1, m_3,\ldots,m_{2p-3}$ are total. By this,
 Definition~\ref{equiv:gj}, and Definition~\ref{def:propmove}, we see that
 $\g_{2p-1}(l , \vec{n}, \nu_2,\nu_4,\ldots, \nu_{2p-2})$ is the \emph{minimum}
  $(2p-1)$-th move of proponent~$\exists$ under the assumption~(i).
 The guessing function $\g_{2p-1}$ is increasing
 with respect to the first argument $l$, because $l$ is the bound of the
 maximization in the definition of $\g_{2p1}-$.
 But there is $n_{2p-1}\in \Nset$ such
 that for every $l$, we have
 $\g_{2p-1}(l,\vec{n},\nu_2,\nu_4,\ldots,\nu_{2p-2}) \le n_{2p-1}$, because of
 \eqref{hosi} and Claim~\ref{equiv:gj}. Therefore the limit
 $m_{2p-1}(\vec{n},\nu_2,\nu_4,\ldots,\nu_{2p-2})$ of
 $\g_{2p-1}(l,\vec{n},\nu_2,\nu_4,\ldots,\nu_{2p-2})$
 with respect to $l$
 is indeed a total function, and actually the limit from
 below. Therefore it is minimum among the possible winning moves. This
 completes the proof of Assertion~\eqref{claim:minimum}.

\medskip
\eqref{assert:mrepresent} The  proof is by induction on $p$. 
Consider the case where $p=1$. Then $m_1(\vec{n})=\lim_l
 \g_1(l,\vec{n}) =\lim_l \mu m_1(\max_{\nu_2<l} g_{k-2}(\vec{n}, m_1,
 \nu_2)) = 0)$. By Claim~\ref{claim:represent}, the total function
 $g_{k-2}$ is represented by some element of $\plim^{k-2}(\A)$.
By Fact~\ref{fact:repr}, the function $\mu m_1(\max_{\nu_2<l} g_{k-2}(\vec{n}, m_1,
 \nu_2)) = 0)$ is represented by some element of $\plim^{k-2}(\A)$. 
By \eqref{equiv:lim}, the function
 $m_1(\vec{n})$ is represented by some element of
 $\plim^{k-1}(\A)$.  Fact~\ref{fact:caninj} implies the function
 $m_1(\vec{n})$ is represented by some element of $\plim^k(\A)$.

Next consider the case where $p>1$.
By Claim~\ref{claim:represent}, a (partial) function $g_{k-2p}$ is
indeed a total function represented by some element of the \textsc{pca}
$\plim^{k-2p}(\A)$.   By applying the bounded
 maximization and then
$\mu$-recursion to  $g_{k-2p}$, define a (partial) function of $l,\vec{n},
x_1,\nu_2,x_3,\nu_4,\ldots,x_{2p-3},\nu_{2p-2}$, as follows
\begin{align}
 \mu m_{2p-1}\left( \max_{\nu_{2p}<l} g_{k-2p} (\vec{n}, x_1, \nu_2, x_3,
 \nu_4,\ldots, x_{2p-3}, \nu_{2p-2}, m_{2p-1}, \nu_{2p})=0 \right). \label{fn:guess}
\end{align} 
Then the (partial) function  is also represented by some element of the \textsc{pca}
$\plim^{k-2p}(\A)$, because of Fact~\ref{fact:repr}.
Let a (partial) function
$F$ of $\vec{n},x_1,\nu_2,x_3,\nu_4,\ldots,x_{2p-3},\nu_{2p-2}$ be guessed
by a (partial) function \eqref{fn:guess} with respect to the variable
$l$. Then $F$ is represented by some element of a \textsc{pca}
$\plim^{k-2p+1}(\A)$ by \eqref{equiv:lim}. By Fact~\ref{fact:caninj},
the function $F$ is represented by some element of a \textsc{pca}
$\plim^k(\A)$.

 By
the induction hypothesis on $p$, all of $(p-1)$ total functions
$m_1(\vec{n})$, $m_3(\vec{n}, \nu_2)$, $m_5(\vec{n}, \nu_2,
\nu_4),\ldots, m_{p-1}(\vec{n}, \nu_2, \nu_4,\ldots, \nu_{2p-4})$ are
represented by some elements of the \textsc{pca} $\plim^k (\A)$.
By composing the $(p-1)$ total functions at the arguments $x_1,x_3,\ldots,x_{2p-3}$ of the
(partial) function $F(\vec{n},x_1,\nu_2,x_3,\nu_4,\ldots,x_{2p-3},\nu_{2p-2})$,
 we obtain
the total function $m_{2p-1}(\vec{n}, \nu_2,\nu_4,\ldots,\nu_{2p-2})$,
according to Definition~\ref{def:propmove}. Thus the total function $m_{2p-1}$ is
represented by some element of the \textsc{pca} $\plim^k(\A)$ by Fact~\ref{fact:repr}.
 \end{proof}

The $\eon$-formula~\eqref{eonformula:hosi} has a realizer $q\in
 \plim^k(\A)$. Here $q$ consists of the following elements of
 $\plim^k(\A)$ : the numerals
 $\num{\vec{n}}$, and the representatives of the total functions $m_1,
 m_3,\ldots$, in view of Claim~\ref{claim:moverepr}.
This completes the proof of Theorem~\ref{thm:eonrealize}.
\end{proof}

From Theorem~\ref{thm:eonrealize}, Theorem~\ref{thm:realPA} follows, by embedding $\ha + \sdne{k}$ in a
corresponding $\eon + (\sdneprime{k+1})$ where $\A$ is a \textsc{pca}.

\subsection{Proofs of Theorem~\ref{thm:fip}
  and Theorem~\ref{thm:smotroaka}}\label{subsec:proofs}

 We prove the non-derivability between the axiom schemes $\fip{k+1}$ and $\sdne{k+1}$~(Theorem~\ref{thm:fip}) by using iterated limiting
 realizability interpretation~(Theorem~\ref{thm:realPA}).
Let $A\; n\; m$ be a $\Pi^0_k$-formula with all the variables indicated. The axiom scheme $\sdne{k + 1}$ proves a sentence
\[
\forall n \left(\neg\neg\exists m.\, A\; n\; m\to \exists m.\,
 A\; n\; m\right).
\]
By this and $\fip{k + 1}$, we derive a sentence $\forall n\exists m.\, \left(\neg\neg\exists m.\, A\; n\; m\to A\; n\; m\right)$.
If the system  $\ha+\sdne{k + 1}+\fip{k + 1}$ is realizable by the
 \textsc{pca} $\plim^k(\Nset)$, then there exists $e\in \Nset$ such that
 for all $n \in\Nset$ the following conditions hold:
\begin{enumerate}
\item $f(n):=\lim_{t_1}\cdots\lim_{t_k} \{ e \}(t_1,\ldots,t_k,n)$ is
       convergent~(In this case, $f$ is $\emptyset^{(k)}$-recursive
      and thus  has a
      $\Pi^0_{k + 1}$-graph); and

\item If $A\;n\;m$ holds for some $m\in\Nset$, then $A\;n\;f(n)$ holds.
\end{enumerate} 
Because $A$ is a $\Pi^0_k$-formula and $f$ has a $\Pi^0_{k + 1}$-graph,
 $A\ n\ f(n)$ is a $\Pi^0_{k + 1}$-relation for $n$. Note that $\exists
 m.\, A\;n\;m$ iff $A\ n\ f(n)$. Because $A$ is an arbitrary
 $\Pi^0_k$-formula, we can choose $A$ such that $\exists m.\, A(\bullet,
 m)$ is a complete $\Sigma^0_{k + 1}$-relation. This contradicts against
 that $A n f(n)$ is a $\Pi^0_{k+1}$-relation. This completes the proof
 of Theorem~\ref{thm:fip}.

Every arithmetical relation $R$ satisfies the \emph{uniformization property}~\cite[]{MR982269}.
That is, if for all natural numbers $n$ there exists a natural number $m$
such that $R(n,m)$, then there exists an arithmetical function $f_R$
such that for all $n$ $R(n,f(n))$.  In Section~\ref{sec:limitingPCA}, we
provide a \textsc{pca} $\plim^\omega(\Nset )$ which represents all such $f_R$'s. In fact, the representative induces a
realizer of $\forall n \exists m.\, R(n,m)$.

 By our prenex normal form theorem~(Theorem~\ref{pnfthm}) and
our iterated limiting realizability
interpretations~(Theorem~\ref{thm:realPA}), we will slightly refine
Smory\'nski's result mentioned in Section~\ref{sec:intro} to
Theorem~\ref{thm:smotroaka}.

\medskip
\emph{Proof of Theorem~\ref{thm:smotroaka}.}
Assume otherwise. By Theorem~\ref{pnfthm},
 for every sentence
$A\in\Gamma$ there is a sentence $\hat A$ in
\textsc{pnf} such that $\hat A$ contains at most $n$ quantifiers and
 $\ha+\Gamma$ proves $A\lequiv \hat
A$.

Since $\ha+\Gamma$ is $n$-consistent, the sentence $\hat A$ in
 \textsc{pnf} is true in the standard model $\omega$. 

First consider the case $\hat A$ is a $\Pi^0_n$-sentence. Then $\hat A$
 can be written as
\begin{align*}
\forall x_1 \exists x_2 \forall x_3 \cdots Q_n x_n.\ R x_1 x_2 x_3 \cdots x_n
\end{align*} 
for some $\Sigma^0_{0}$-formula $R$.

Here $\forall x_3 \exists x_4 \cdots Q_n x_n.\ R x y x_3 \cdots x_n
$ defines a $\emptyset^{(n-2)}$-recursive  binary
 relation on $\omega$. By the relativization of the \emph{uniformization property for recursive
relations}~\cite[]{MR982269}, there exists some
 $\emptyset^{(n-2)}$-recursive function
\begin{align*}
f_2(x):=\mu y. \forall x_3 \cdots Q_n x_n.\ R x y x_3 \cdots x_n
\end{align*}
 such that
for each natural number $x_1$ a formula
$\forall x_3 \exists x_4 \cdots Q_n x_n.\ R x_1 f_2(x_1) x_3\cdots x_n$ is true
 on $\omega$.

 In this way, there are
 $\emptyset^{(n-2)}$-functions $f_2(x_1), f_4(x_1, x_3),\ldots$ such
 that 
\begin{align*}
\forall x_1 \forall x_3 \forall x_5  \cdots.\ R\; x_1\; f_2(x_1)\; x_3
\; f_4(x_1, x_3)\; x_5 \cdots
\end{align*} 
 holds on $\omega$. 

If  $\hat A$ is not a $\Pi^0_n$-sentence, then  $\hat A$ 
is written as
$\exists x_1 \forall x_2 \exists x_3 \cdots Q_n x_n.\ R x_1 x_2 x_3 \cdots x_n
$. Then there are natural number $n_1$ and
 $\emptyset^{(n-3)}$-recursive functions $f_3(x_2), f_5(x_2, x_4),\ldots$
 such that a formula $\forall x_2 \forall x_4\forall x_6 \cdots  R n_1 x_2 f_3(x_2)
 x_4 f_5(x_2, x_4)\cdots $ holds 
 on $\omega$. 

Because a \textsc{pca}
 $\plim^n(\Nset)$ can represent all the
$\emptyset^{(n)}$-functions $f_i$'s, we can find a realizer of $\hat
A$ in the \textsc{pca} $\plim^n(\Nset)$. 

 The \textsc{pca} $\plim^n (\Nset)$ realizes $\slem{n}$, and thus the formula
 $A\in\Gamma$ by Theorem~\ref{pnfthm}. Because the \textsc{pca} $\plim^n (\Nset)$ does not realize
 $\slem{n+1}$, we conclude $\ha+\Gamma\not\vdash\slem{n+1}$. This completes the proof of Theorem~\ref{thm:smotroaka}.
\bigskip

Our use of the complete set $\emptyset^{(n)}$ contrasts against Kleene's
use of \emph{extended Church's thesis} on defining \emph{effectively
true}~(\emph{general recursively true}) prenex normal form~(see
Section~79 of \cite{MR0051790}).

\cite{MR717260}  considered other versions  $HA$  and $PA$ of Heyting's
 arithmetic and Peano's
 arithmetic, where $HA$ and $PA$ are formalized by the language
\begin{align*}
\{0,1,2,3,\ldots;\ Z(,),\; S(,),\; A(,,),\; M(,,),\; =\},
\end{align*}
 and then proved ``Let $\Gamma$ be a set of sentences of
bounded quantifier-complexity, and suppose $HA+\Gamma\vdash PA$. Then
$HA+\Gamma$ is inconsistent.'' For the proof, assuming otherwise, Smory\'nski
constructed a model of $PA$ by applying Orey's compactness theorem  to
 $HA+\Gamma$. For Orey's compactness theorem, see
Chapter~4 of \cite{smorynski78:_nonst_model_of_arith}, \cite{MR0146082},
 \cite{MR1748522} and Theorem.~III~2.39~(i)$\iff$(ii) of \cite{MR1748522}.
 Then he constructed a Kripke model~(see Section~5.2.3 of
\cite{MR48:3699}) for $HA$ to derive the contradiction. See
\cite{MR717260} for a proof formalized within a formal system  $PA$ + $1$-Con($PA$). 

However, the referee wrote 
\begin{quote}\emph{
``As far as I can see Smory\'nski leaves open
whether there can be a consistent, classically unsound, finite extension
of $\ha$ that implies full sentential excluded third. I definitely do
believe there isn't. It is unknown whether the analogous result holds for all classically
 invalid constructive propositional schemes.''}
\end{quote}
The author cannot help but suppose that the language of the $\ha$
referee meant consists of the function symbols for all the primitive
recursive functions and the identity predicate.  It may be important to
construct Kripke models of such $\ha$ by employing model theory of
arithmetic.  The author thinks the referee's last sentence suggests a
possible research direction.

As in the proof of Theorem~\ref{thm:eonrealize}, we hope that the
wording ``game,'' ``strategy,'' ``move,'' and so on are useful to
explain realizability interpretation neatly, and that various
realizability interpretations of logical principles over $\ha$ are
related to  circumstances where one or the other player of
a various game have a winning strategy, and the consequences of the
existence of such strategies.

\section*{Acknowledgement}  The author acknowledges Susumu Hayashi,
Pieter Hofstra,
      Stefano Berardi, an anonymous referee and Craig Smory{\'n}ski. The
      anonymous referee informed the author of Smory{\'n}ski's work and
      Smory{\'n}ski let the author know his course notes.


\begin{thebibliography}{100}

\bibitem[Akama, 2004]{MR2030297}
Akama, Y. (2004).
\newblock Limiting partial combinatory algebras.
\newblock {\em Theoret. Comput. Sci.}, 311(1-3):199--220.

\bibitem[Akama et~al., 2004]{akama04:_arith_hierar_of_laws_of}
Akama, Y., Berardi, S., Hayashi, S., and Kohlenbach, U. (2004).
\newblock An arithmetical hierarchy of the laws of excluded middle and related
  principles.
\newblock In {\em Proceedings of the 19th Annual {IEEE} Symposium on Logic in
  Computer Science}, pages 192--201.


\bibitem[Avigad, 2000]{avigad00}
Avigad, J. (2000).
\newblock Realizability interpretation for classical arithmetic.
\newblock In Buss, H\'ajek, and Pudl\'ak, editors, {\em Logic Colloquium '98},
  number~13 in Lecture Notes in Logic, pages 57--90. AK Peters.

\bibitem[Beeson, 1985]{MR786465}
Beeson, M.~J. (1985).
\newblock {\em Foundations of constructive mathematics}, volume~6 of {\em
  Ergebnisse der Mathematik und ihrer Grenzgebiete (3) [Results in Mathematics
  and Related Areas (3)]}.
\newblock Springer-Verlag, Berlin.
\newblock Metamathematical studies.

\bibitem[Berardi, 2005]{MR2121494}
Berardi, S. (2005).
\newblock Classical logic as limit completion.
\newblock {\em Math. Structures Comput. Sci.}, 15(1):167--200.

\bibitem[Berardi et~al., 1998]{berardi98computational}
Berardi, S., Bezem, M., and Coquand, T. (1998).
\newblock On the computational content of the axiom of choice.
\newblock {\em J. Symbolic Logic}, 63(2):600--622.


\bibitem[H{\'a}jek and Pudl{\'a}k, 1998]{MR1748522}
H{\'a}jek, P. and Pudl{\'a}k, P. (1998).
\newblock {\em Metamathematics of first-order arithmetic}.
\newblock Perspectives in Mathematical Logic. Springer-Verlag, Berlin.
\newblock Second printing.

\bibitem[Hayashi et~al., 2002]{Hayashi01:_towar_animat_proof}
Hayashi, S., Sumitomo, R., and Shii, K. (2002).
\newblock Towards animation of proofs -- testing proofs by examples --.
\newblock {\em Theoret. Comput. Sci.}, 272(1--2):177--195.

\bibitem[Hindley and Seldin, 1986]{BlueBook}
Hindley, J. and Seldin, J. (1986).
\newblock {\em Introduction to Combinators and Lambda-calculus}.
\newblock Cambridge University Press.

\bibitem[Hirschfeld, 1975]{MR0381969}
Hirschfeld, J. (1975).
\newblock Models of arithmetic and recursive functions.
\newblock {\em Israel J. Math.}, 20(2):111--126.

\bibitem[Hofstra and Cockett, 2010]{Hofstra:2010:UTU:1857260.1857303}
Hofstra, P. and Cockett, R. (2010).
\newblock Unitary theories, Unitary categories.
\newblock {\em Electronic Notes in Theoretical Computer Science},  265:11--33.


\bibitem[Kleene, 1945]{MR0015346}
Kleene, S.~C. (1945).
\newblock On the interpretation of intuitionistic number theory.
\newblock {\em J. Symbolic Logic}, 10:109--124.

\bibitem[Kleene, 1952]{MR0051790}
Kleene, S.~C. (1952).
\newblock {\em Introduction to metamathematics}.
\newblock D. Van Nostrand Co., Inc., New York, N. Y.

\bibitem[Lerman, 1970]{MR0265157}
Lerman, M. (1970).
\newblock Recursive functions modulo {${\rm CO}-r$}-maximal sets.
\newblock {\em Trans. Amer. Math. Soc.}, 148:429--444.


\bibitem[Odifreddi, 1989]{MR982269}
Odifreddi, P. (1989).
\newblock {\em Classical recursion theory}, volume 125 of {\em Studies in Logic
  and the Foundations of Mathematics}.
\newblock North-Holland Publishing Co., Amsterdam.
\newblock The theory of functions and sets of natural numbers, With a foreword
  by G. E. Sacks.

\bibitem[Orey, 1961]{MR0146082}
Orey, S. (1961).
\newblock Relative interpretations.
\newblock {\em Z. Math. Logik Grundlagen Math.}, 7:146--153.


\bibitem[Smory{\'n}ski, 1978]{smorynski78:_nonst_model_of_arith}
Smory{\'n}ski, C. (1978).
\newblock Nonstandard models of arithmetic.
\newblock Course notes, fall 1978,  Utrecht University.
\newblock Logic group preprint series.

\bibitem[Smory{\'n}ski, 1982]{MR717260}
Smory{\'n}ski, C. (1982).
\newblock Nonstandard models and constructivity.
\newblock In {\em The {L}. {E}. {J}. {B}rouwer {C}entenary {S}ymposium
  ({N}oordwijkerhout, 1981)}, volume 110 of {\em Stud. Logic Found. Math.},
  pages 459--464. North-Holland, Amsterdam.

\bibitem[Troelstra, 1973]{MR48:3699}
Troelstra, A.~S., editor (1973).
\newblock {\em Metamathematical investigation of intuitionistic arithmetic and
  analysis}.
\newblock Springer-Verlag, Berlin.
\newblock Lecture Notes in Mathematics, Vol. 344.
\end{thebibliography}
\end{document}